\documentclass[11pt,a4paper]{amsart}

\usepackage[a-3u]{pdfx}

\usepackage[T1]{fontenc}
\usepackage{lmodern,ae,aecompl}

\usepackage{amsmath,amsfonts,amsthm,amssymb}
\usepackage{mathtools}

\usepackage{thmtools}
\usepackage{thm-restate}

\usepackage{graphicx}
\usepackage[inline]{enumitem}
\usepackage[backend=biber,articlein=false,style=ext-numeric,sorting=nty,doi=false,isbn=false,url=false,eprint=false,giveninits=true]{biblatex}

\usepackage{hyperref}\hypersetup{unicode,breaklinks=true,colorlinks=true,allcolors=blue}
\usepackage{cleveref}

\DeclareMathOperator{\Lip}{Lip}

\DeclareMathOperator{\diam}{diam}

\usepackage{macros_nagata}

\begin{document}

\title{Nagata Dimension and Lipschitz Extensions Into Quasi-Banach Spaces}
\begin{abstract}
    Given two metric spaces $\mathcal N \subseteq \mathcal M$ in inclusion and $0<p\leq 1$, we wish to determine the smallest constant $\mathfrak{t}_p (\mathcal N, \mathcal M)$ such that any Lipschitz map $f: \mathcal N \to Z$ into any $p$-Banach space $Z$ can be extended to a Lipschitz map $f' : \mathcal M \to Z$ satisfying $\Lip f' \leq \mathfrak{t}_p (\mathcal N, \mathcal M)\cdot \Lip f$. In this article, we prove that if $\mathcal N$ has finite Nagata dimension at most $d$ with constant $\gamma$, then $\mathfrak{t}_p (\mathcal N, \mathcal M) \lesssim_p \gamma \cdot (d+1)^{1/p -1} \cdot \log (d+2)$ for all $0<p\leq 1$. We show that examples of spaces with finite Nagata dimension include doubling spaces, as well as minor-excluded metric graphs. We also establish that the constant $\mathfrak{t}_p (\mathcal N, \mathcal M)$ generally increases as $p$ approaches zero.
\end{abstract}
\author{Jan B\'ima}

\email{jan.bima@mff.cuni.cz}
\address{Charles University, Faculty of Mathematics and Physics, Department of Mathematical Analysis, Sokolovsk\'a 83, 186 75 Prague 8, Czech Republic}
\thanks{I would like to express my gratitude to M. Cúth for his numerous valuable insights into my research efforts. I acknowledge the support of GAČR 23-04776S and of the Charles University project GA UK No. 138123. The work was completed while the author was an employee at MSD Czech Republic, Prague, Czech Republic.}
\maketitle

\section{Introduction}\label{sec:introduction}

Suppose that $\mathcal M$ and $\mathcal T$ are metric spaces and $\mathcal N$ is a non-empty subset of $\mathcal M$. A classical question in metric space theory, known as the \emph{Lipschitz extension problem}, asks what the smallest possible constant $C$ is, such that any Lipschitz map $f: \mathcal N \to \mathcal T$ can be extended to $\widetilde f : \mathcal M \to \mathcal T$, where $\Lip \widetilde f \leq C \Lip f$. Owing to its fundamental importance and widespread application in geometry and approximation theory, this problem has received significant attention. The literature on the subject is extensive, with notable contributions including those by \textcite{Kirszbraun1934,Johnson1986,Ball1992,Lee2004}, to name a few.

In this paper, our focus is on two specific kinds of Lipschitz extension problems, namely the \emph{trace} and \emph{absolute extendability problems}. For any two metric spaces $\mathcal N \subseteq \mathcal M$ in inclusion and a family of metric space $\mathcal F$, we define the \emph{trace} $\mathfrak{t}_{\mathcal F} (\mathcal M, \mathcal N)$ as the infimum over constants $C\in (0, \infty]$ where any Lipschitz map $f : \mathcal N \to \mathcal T$ for any $\mathcal T \in \mathcal F$, can be extended to $\widetilde f : \mathcal M \to \mathcal T$ such that $\Lip \widetilde f \leq C \Lip f$. Subsequently, the \emph{absolute extendability constant} $\mathfrak{ae}_{\mathcal F} (\mathcal N)$ is the supremum of traces $\mathfrak{ae}_{\mathcal F} (\mathcal M, \mathcal N)$ across all metric spaces $\mathcal M \supseteq \mathcal N$. That is, $\mathfrak{ae}_{\mathcal F} (\mathcal N) = \sup \{ \mathfrak{ae}_{\mathcal F} (\mathcal N, \mathcal M) : \mathcal N \subseteq \mathcal M\}$.

The family of all Banach spaces is most commonly taken as $\mathcal F$. In this setting, the classes of absolute extendable metric spaces have been identified and extensively investigated, along with the associated trace problems, by \textcite{Johnson1986,Matousek1990,Lee2004,Lang2005,Brudnyi2007,Naor2017,Basso2022}. More recent research has explored extensions of Lipschitz maps ranging into quasi-metric and quasi-Banach spaces, as detailed by \textcite{Basso2018} and \textcite{Albiac2021sums}, respectively. This paper also contributes to this ongoing line of research. Interestingly, establishing Lipschitz extendability results in the quasi-metric setting often necessitates innovative proof techniques. These novel approaches, in turn, provide a new perspective on the case where Banach spaces are considered.

Before presenting our results in greater detail, we recall the concept of quasi-Banach spaces and properly adopt the definitions of trace and absolute extendability. In what follows, the notation $A\lesssim B$ means that $A \leq C B$ for some universal constant $C\geq 0$.

To that end, we recall $(X, \lVert \cdot \rVert)$ is called a \emph{quasi-normed space} if $X$ is a vector space and $\lVert \cdot \rVert$ is a quasi-norm, that is,
\begin{enumerate*}[label=(\roman*), ref=\roman*]
    \item $\lVert x \rVert > 0$ for any $x\neq 0$,
    \item\label{it:quasi_norm_homogeneity} $\lVert \alpha x \rVert = |\alpha|\lVert x \rVert$ for any scalar $\alpha$ and $x\in X$,
    \item\label{it:quasi_norm} and $\lVert x+y\rVert \lesssim \lVert x\rVert + \lVert y\rVert$ for any $x,\,y\in X$.
\end{enumerate*}
We call $X$ a \emph{quasi-Banach} space if it is complete with respect to the linear metric topology induced by $\lVert\cdot\rVert$.

It turns out that every quasi-Banach space is isomorphic to a \emph{$p$-Banach space} for some $0<p\leq 1$, as shown in \cite[Theorem~1.2]{Kalton1984}. The converse is trivially true.

\begin{defn}
    Let $X$ be a vector space and $0<p\leq 1$. We say that a map $\lVert \cdot \rVert : X\to [0, \infty)$ is a \emph{$p$-norm} on $X$ if, in addition to \cref{it:quasi_norm_homogeneity,it:quasi_norm},
    \begin{enumerate}[start=3,label=(\roman*'), ref=\roman*']
        \item\label{it:p_norm} $\lVert x+y\rVert^p \leq \lVert x\rVert^p + \lVert y\rVert^p$ for any $x,\,y\in X$.
    \end{enumerate}
    
    We then call $(X, \lVert\cdot\rVert)$ a \emph{$p$-normed} space. If $X$ is complete with respect to the metric $d(x,y)=\lVert x-y\rVert^p$, where $x,\,y\in X$, we say $(X, \lVert\cdot\rVert)$ is a \emph{$p$-Banach space}.
\end{defn}

Developing extendability results for maps ranging into any general quasi-Banach space would be overly ambitious. Therefore, we adopt the following definition.

\begin{defn}
    For a metric space $\mathcal N$ and each $0<p\leq 1$, we define the \emph{$p$-trace $\mathfrak{t}_p (\mathcal N, \mathcal M)$ of $\mathcal N$ in $\mathcal M \supseteq \mathcal N$} to be the infimum over all $C\in (0, \infty]$ such that for any $p$-Banach space $Z$, any Lipschitz map $f : \mathcal N \to Z$ has a Lipschitz extension $\widetilde f : \mathcal M \to Z$ with $\Lip \widetilde f \leq C \Lip f$. 
    
    We define the \emph{absolute $p$-extandability constant} $\mathfrak{ae}_p (\mathcal N) = \sup \{\mathfrak{t}_p (\mathcal N, \mathcal M) : \mathcal M \supseteq \mathcal N \}$. If $\mathfrak{ae}_p (\mathcal N) < \infty$, we say that $\mathcal N$ is \emph{absolutely $p$-extendable}.
\end{defn}

To better motivate the notions we just introduced, particularly in relation to the $p$-Banach setting where $0<p\leq 1$, we can consider extensions from finite subsets. Note that it is easy see that finite metric spaces are absolutely $p$-extendable. However, determining the lower and upper estimates on the extendability constant $\mathfrak{ae}_1 (n) = \sup \{ \mathfrak{ae}_1(\mathcal N): |\mathcal N| \leq n \}$ is a significant open problem, see \textcite{Lee2004absolute,Naor2017}. In general, the problem becomes even more challenging for $0<p<1$. While it can easily be observed that $\mathfrak{ae}_p (\mathcal N) = 1$ for any two-point metric space $\mathcal N$ and $0<p<1$, the absolute extendability constant typically increases as $p$ approaches zero. In particular, we have the following result, which is intriguing when compared to \cite[Theorem 1.1]{Basso2018}, asserting that $\mathfrak{t}_p (\mathcal N, \mathcal M) \leq m+1$ whenever $|\mathcal M \setminus \mathcal N| \leq m$, for any $m\in \mathbb N$ and all $0<p\leq 1$.

\begin{introthm}[{cf.~\Cref{thm:t_p_counterexample}}]\label{introthm:t_p_counterexample}
    Let $\mathcal N = \{0, 1, 2\} \subseteq (\mathbb R, |\cdot|)$ and $\mathcal M = \mathcal N \cup \{ 3/2 \}$. Then $\mathfrak{t}_1 (\mathcal N, \mathcal M) = 1$ but $\mathfrak{t}_p (\mathcal N, \mathcal M) > 1$ for any $0<p<1$. Moreover, we have $\mathfrak{t}_p (\mathcal N, \mathcal M) \to 2$ as $p\to 0$.
\end{introthm}

Having motivated the notion of absolute extandability, let us review the following series of results on non-trivial families of absolutely 1-extendable spaces, which will have a significant role in the sequel.

\begin{thm}[{\textcite{Johnson1986}}]
    If $\mathcal N$ is a subset of an $n$-dimensional normed vector space $Y$, then $\mathfrak{t}_1 (\mathcal N, Y) \lesssim n$.
\end{thm}

Observe that, as a straightforward corollary, we have $\mathfrak{ae} (n) \lesssim n$ for all $n\in\mathbb N$. Indeed, it suffices to observe that for any $n$-point metric space $\mathcal N$ and $\mathcal M \supseteq \mathcal N$, there exists a non-expansive map of $\mathcal M$ into $\ell_\infty^n$ which, moreover, is an isometry on $\mathcal N$.

The proof presented in \cite{Johnson1986} relies on a specific Whitney-type cover of the ambient space $Y \supseteq \mathcal N$, with its existence being facilitated by the presence of the Lebesgue measure on $Y$. Abstracting this approach, the only necessary condition for the proof to pass through is that the space $\mathcal M$ fulfills the \emph{doubling} property.

\begin{defn}
    We say that a~metric space $(\mathcal N, \rho)$ is \emph{doubling} if there exists $\lambda_{\mathcal N} \in \mathbb N$, called the \emph{doubling constant}, such that any closed ball in $\mathcal N$ of radius $2r>0$ can be covered by $\lambda_{\mathcal N}$-many closed balls of radius $r$.
\end{defn}

Examples of doubling metric spaces are provided by subsets of Carnot groups, see \cite{LeDonne2017}, and, in particular, subsets of finite-dimensional spaces. To that end, we recall that whenever $\mathcal N$ is a~subset of an~$n$-dimensional normed space, we have $\log \lambda_{\mathcal N} \lesssim n$, as shown in \cite[Theorem 3]{Rogers1963}.

As an application of general extension results obtained using the method of \emph{stochastic metric decomposition}, \textcite{Lee2004} achieved the following.

\begin{thm}[{\textcite{Lee2004}}]\label{thm:ae_doubling}
    If $\mathcal N$ is a doubling metric space, then $\mathfrak{ae}_1 (\mathcal N) \lesssim \log \lambda_{\mathcal N}$.
\end{thm}

Recently, a generalization of the result was addressed for maps ranging into $p$-Banach spaces by \textcite{Albiac2021sums}.

\begin{thm}[{\textcite{Albiac2021sums}}]\label{thm:ae_p_doubling}
    If $\mathcal N$ is a doubling metric space, then $\mathfrak{ae}_p (\mathcal N) \lesssim (15\lambda_{\mathcal N}^4)^{1/p}$ for all $0<p\leq 1$.
\end{thm}

The proof of \Cref{thm:ae_p_doubling} uses a distinct method in comparison to that of \Cref{thm:ae_doubling}. Specifically, the technique of stochastic decompositions as applied in \cite{Lee2004} results in an extension map being defined by a Bochner integral, specifically over a distributional partition of unity. This approach does not readily adapt to the quasi-metric setting. Meanwhile, the estimate from \Cref{thm:ae_p_doubling} is evidently suboptimal for $p=1$ in light of \Cref{thm:ae_doubling}, raising the question of whether the approach can be improved. Revisiting the ideas used in the original result due to \textcite{Johnson1986}, we show that this indeed is the case.

\begin{notation}
    In what follows, the notation $A\lesssim_p B$ means that $A \leq C(p)\cdot B$ for some constant $C(p)\geq 0$ dependent only on $p$.
\end{notation}

\begin{introthm}[{\emph{for doubling spaces}; cf.~\Cref{cor:ae_p_doubling}}]\label{introthm:ae_p_doubling}
    If $\mathcal N$ has doubling constant $\lambda_N > 1$, then $\mathfrak{ae}_p (\mathcal N) \lesssim_p \lambda_{\mathcal N}^{3/p-3} \cdot \log \lambda_{\mathcal N}$ for all $0<p\leq 1$.
\end{introthm}

Importantly, we show that the assumptions on $\mathcal N$, necessary for the existence of a specific Whitney-type cover, can actually be further generalized. More precisely, it suffices to assume that $\mathcal N$ has a finite \emph{Nagata dimension}. Notably, the resulting quantitative estimate is a new result even in the Banach setting where $p=1$.\footnote{We remark that this result has been recently independently discovered by \textcite{Basso2023lipschitz}, in the narrower context of the Banach setting with $p=1$. See also \Cref{rem:basso}.} To this end, we recall an earlier estimate showing that if $\mathcal N$ has Nagata dimension at most $d$ with constant $\gamma$, then $\mathfrak{ae}_1 (\mathcal N) \lesssim \gamma d^3$, see \cite[Corollary 5.2]{Naor2011}.

\begin{defn}
    Let $(\mathcal N, \rho)$ be a metric space. Given $\gamma \geq 1$ and $d\in\mathbb N_0$, we say that $\mathcal N$ has \emph{Nagata dimension at most $d$ with constant $\gamma$}, if for every $s>0$, there exists a family $\mathcal C$ of non-empty subsets in $\mathcal N$ with the following properties:
    \begin{enumerate}[label=(\alph*), ref=\alph*]
        \item\label{it:nagata_1} $\mathcal C$ covers $\mathcal N$, i.e., $\bigcup_{C\in\mathcal C} C = \mathcal N$,
        \item\label{it:nagata_2} for every $C\in \mathcal C$, $\diam C \leq \gamma s$,
        \item\label{it:nagata_3} for every $A\subseteq \mathcal N$ with $\diam A \leq s$, we have $| \{ C\in\mathcal C : C \cap A \neq \emptyset \} | \leq d +1 $.
    \end{enumerate}
\end{defn}

It is known that if a metric space $\mathcal N$ has a doubling constant $\lambda_{\mathcal N}$, it has Nagata dimension at most $\lambda_{\mathcal N}^3-1$ with constant 2 (see \Cref{lemma:doubling_nagata}). The Nagata dimension can in fact be bounded by the logarithm of the doubling constant (see \cite{LeDonne2015}, where an alternative definition of the doubling dimension, referred to as the \emph{Assouad dimension}, is considered). In our applications, however, this gain would be compensated by increased values of the constant $\gamma$.

\addtocounter{introthm}{-1}
\begin{introthm}[{\emph{for spaces with finite Nagata dimension}; cf.~\Cref{thm:ae_p_nagata}}]
    If $\mathcal N$ has Nagata dimension at most $d$ with constant $\gamma$, then $\mathfrak{ae}_p (\mathcal N) \lesssim_p \gamma \cdot (d+1)^{1/p -1} \cdot \log (d+2)$ for all $0<p\leq 1$.
\end{introthm}

An important class of spaces with finite Nagata dimension is constituted by \emph{metric graphs}, which are one-dimensional simplicial complexes induced by weighted graphs (see~\Cref{def:metric_graph}). It can be shown that if a weighted graph $G=(V, E)$ excludes the complete graph on $m$ vertices, $K_m$, as a minor, then the Nagata dimension, together with constant $\gamma$ of the metric graph $\Sigma (G)$, can be bounded in terms of $m$. We note that for minor-excluded unweighted graphs, the finiteness of the Nagata dimension was established in \cite[Theorem 2.2]{Ostrovskii2015}. In \Cref{prop:nagata_minor_excluded}, we generalize this result to the class of all metric graphs.

We remark that a trace theorem for subsets of metric trees, which are metric graphs induced by weighted trees, was established in \textcite{Matousek1990}. More recently, this result has been generalized to an absolute extendability result for subsets of metric graphs, induced by minor-excluded weighted graphs, in \textcite[Theorem 5.1]{Lee2004}. By aplying \Cref{prop:nagata_minor_excluded} in conjunction with the general extension result for spaces with finite Nagata dimensions, as established in \Cref{introthm:ae_p_doubling}, we present an alternative proof for this generalized result and adapt it to the $p$-Banach setting.

\begin{introthm}[{cf.~\Cref{thm:ae_p_graphs}}]\label{introthm:ae_p_minor_excluded}
    If $\Sigma(G)$ is a metric graph induced by a countable and connected weighted graph $G$ which excludes the complete graph $K_m$ as a minor, then $\mathfrak{ae}_p (S) \lesssim_p m^2 \cdot 9^{m(1/p - 1)}$ for any subset $S\subseteq \Sigma(G)$ and any $0<p\leq 1$.
\end{introthm}

\Cref{introthm:ae_p_doubling} has an important application to the theory of \emph{Lipschitz free $p$-spaces}. We note that these spaces were first considered by \textcite{Albiac2009}, and their systematic study was initiated by \textcite{Albiac2020} and has continued through \cite{Albiac2021sums,Albiac2022,Bima2023}. For a given metric space $\mathcal M$ and $0<p\leq 1$, there exists a unique, up to an isomoprhism, $p$-Banach space $\mathcal F_p (\mathcal M)$, called the \emph{Lipschitz free space over $\mathcal M$}, such that $\mathcal M$ embeds isometrically into $\mathcal F_p(\mathcal M)$ via a map $\delta:\mathcal M \to \mathcal F_p(\mathcal M)$, and for every $p$-Banach space $Y$ and a Lipschitz map $f:\mathcal M\to Y$ which vanishes at the origin, $f$ extends uniquely to a linear operator $T_f : \mathcal F_p(\mathcal M)\to Y$ such that $\Lip f = \lVert T_f\rVert$.

It was an open question whether, for metric spaces $\mathcal N \subseteq \mathcal M$ and $0<p\leq 1$, the linearization $T_i : \mathcal F_p (\mathcal N) \to \mathcal F_p (\mathcal M)$ of the canonical inclusion $i: \mathcal N \to \mathcal M$ is an isomorphism (see \cite[Question 6.2]{Albiac2020}). As an essential part, \Cref{introthm:ae_p_doubling} (see also \Cref{cor:ae_p_trees} or apply \Cref{introthm:ae_p_minor_excluded} with $m=3$) was used in \cite[Theorem 3.21]{Cuth2023} to show that this is indeed the case. Specifically, the value $\lVert T_i^{-1} \rVert$ can be bounded by the absolute $p$-extendability constant of a metric tree, that is, a graph which excludes $K_3$ as minor.

This paper is structured as follows. In \Cref{sec:whitney_covers}, we lay out a general extension result which only relies on the existence, and quality, of a particular Whitney-type cover on the space $\mathcal M \setminus \mathcal N$. The presence of such covers for spaces with finite Nagata dimension is established in Section \Cref{sec:nagata_covers}. \Cref{sec:lip_free_spaces} introduces the counterexample from \Cref{introthm:t_p_counterexample}. The proof will rely on the properties of Lipschitz free $p$-spaces, and in this context, we will discuss the general connection of the Lipschitz extension problem to Lipschitz free $p$-spaces. In \Cref{sec:open_problems}, we formulate two open questions related to the dependence of the Lipschitz extension constant on $p$.

\section{Covers of the Whitney Type and Lipschitz Extensions}\label{sec:whitney_covers}

In order to extend Lipschitz maps defined on $\mathcal N$ to the ambient space $\mathcal M$, we consider a partition of unity $\{ \phi_i \}_{i\in\mathcal I}$ induced by a cover $\{ K_i \}_{i\in\mathcal I}$ on $\mathcal M \setminus \mathcal N$. For this purpose, we adopt the general concept of a Whitney-type cover; thus, we cover the space $\mathcal M \setminus \mathcal N$ in a way that the diameters of the covering sets are proportional to their distance from the set $\mathcal N$.

\begin{defn}
    Let $\mathcal N$ be a closed subset of a metric space $(\mathcal M, \rho)$. We say that a family $\{K_i\}_{i\in \mathcal I}$ of open sets in $\mathcal M\setminus \mathcal N$ is a~\emph{Whitney cover of $\mathcal M\setminus \mathcal N$ with parameters $(o, s, d, a)$}, where $o\in\mathbb N$ and $s,\, d,\, a>0$, provided
    \begin{enumerate}[label=(\roman*), ref=\roman*]
        \item\label{it:whitney_1} for any $x\in \mathcal M\setminus \mathcal N$, it holds $|\{ i \in \mathcal I:x\in K_i\}|\leq o$,
        \item\label{it:whitney_2} for any $x\in \mathcal M\setminus \mathcal N$, there exists $i\in\mathcal I$ such that $\rho(x, \mathcal M\setminus K_i)\geq s\cdot\rho(x, \mathcal N)$,
        \item\label{it:whitney_3} it holds that $\diam K_i \leq d\cdot\rho(K_i, \mathcal N)$ for any $i\in\mathcal I$,
        \item\label{it:whitney_4} for any $i\in I$ and $x,\,y\in K_i$, we have $\rho(x, \mathcal N) / \rho(y, \mathcal N) < a$.
    \end{enumerate}
\end{defn}

Observe that by \cref{it:whitney_2}, $\{ K_i \}_{i\in\mathcal I}$ is, in particular, a cover of $\mathcal M\setminus \mathcal N$.

Note that we do not require the covering sets to be disjoint. However, the \emph{overlapping constant} proves to be a significant factor for the quantitative estimates on the Lipschitz constant of resulting extensions. This is so because, for a given point $y\in\mathcal M \setminus\mathcal N$, the extension of a map $f$ will be defined as $\sum_{i\in\mathcal I} \phi_i(y) f(x_i)$, for some predetermined $x_i \in\mathcal N$. Note that in the setting where $p<1$, there arises an additional quasimetric factor (see \cref{it:p_norm}), which grows as $n^{1/p-1}$ for $n$ summands. For brevity, we will introduce the following notation.

\begin{notation}
    We set $C(p, n)= n^{1/p -1}$ where $0<p\leq 1$ and $n\in\mathbb N$. Note that for any $p$-normed space $(X, \lVert\cdot\rVert_p)$ and $x_i\in X$, where $i\in \{1,\ldots, n\}$, we have $\lVert \sum_{i=1}^n x_i \rVert_p \leq C(p, n)\sum_{i=1}^n \lVert x_i \rVert_p$.
\end{notation}
    
This reasoning also explains why we introduce another type of cover that is distinct from the one recently considered by \textcite{Lee2004}. Mainly, the authors examine \emph{distributions} over covers and derive 'well-behaved' partitions of unity where desirable properties are typically attained "on average." While it is often possible to pass to locally finite partitions, we find it more convenient for our purposes to construct the covers in a way that lets us more precisely track the overlapping constant.

Assuming the existence of Whitney cover on $\mathcal M \setminus \mathcal N$, we can now construct a Lipschitz extension to the space $\mathcal M$. Let us note that the adapted definition of Whitney-type cover, yet without the \cref{it:whitney_4}, appeared in \cite[Proposition 5.2]{Albiac2021sums}, and a related result on Lipschitz extension for maps ranging into $p$-Banach spaces was given in \cite[Theorem 5.1]{Albiac2021sums}. We remark that in the referenced paper, the extension problem was examined within the context of Lipschitz free $p$-spaces. Here we achieve to refine the extension result further by employing an optimization trick originating from the work of \textcite{Johnson1986}.

\begin{thm}\label{thm:lipschitz_extension_p}
    Let $\mathcal N \subset (\mathcal M, \rho)$ be non-empty and such that $\mathcal M\setminus \mathcal N$ admits a~Whitney cover with parameters $(o, s, d, a)$, where $o\in\mathbb N$, $o>1$, and $s,\, d>0$. 
    Then for any $0<p\leq 1$, the $p$-trace $\mathfrak{t}_p (\mathcal N, \mathcal M)$ of $\mathcal N$ in $\mathcal M$ is at most $D(p, s, d, a)\cdot C(p, o) \log_2 (2o)$, where $D(p, s, d, a)$ is a universal constant, quantified in \eqref{eq:const_D}.
\end{thm}

Let us first note the following elementary inequality.

\begin{lemma}\label{lemma:ineq_means}
    For any $m\geq 1$, $n\in\mathbb N$, and $a_i\geq 0$, where $i\in\{1,\ldots, n\}$ and at least one $a_i$ is non-zero, it holds that \[\sum_{i=1}^n a_i^{m-1}/\sum_{i=1}^n a_i^m\leq n^{1/m}/\left(\sum_{i=1}^n a_i^m\right)^{1/m} \text. \]
\end{lemma}

\begin{proof}
    Note that $x\mapsto x^{(m-1)/m}$ is concave on $\mathbb R^+$. Consequently, we get $\sum_{i=1}^n a^{m-1}_i/n \leq \left(\sum_{i=1}^n a_i^m/n\right)^{(m-1)/m}$. Rearranging the terms, the claim is established.
\end{proof}

\begin{proof}[Proof of~\Cref{thm:lipschitz_extension_p}.]
    Let $\{ K_i\}_{i\in\mathcal I}$ be a Whitney cover for $\mathcal M\setminus \mathcal N$ with parameters $(o, s, d, a)$. We can assume that each set $K_i$, where $i\in\mathcal I$, is non-empty. Also, we pick $0<p\leq 1$ and consider a parameter $m\geq 1$, to be optimized later.
    
    We construct a partition of unity $\{ \phi_i \}_{i\in\mathcal I}$ on $\mathcal M \setminus \mathcal N$ as follows. For every $i\in \mathcal I$, we define a~map $\phi'_i : \mathcal M\setminus \mathcal N \to \mathbb R$ as $x \mapsto \rho^m(x, \mathcal M\setminus K_i)$, where $x\in \mathcal M\setminus \mathcal N$. It follows from \cref{it:whitney_1,it:whitney_2} that for any $x\in \mathcal M\setminus \mathcal N$, the set $\{i\in\mathcal I : \phi'_i(x)>0\}$ is finite non-empty. Consequently, we set $\phi_i : \mathcal M\setminus \mathcal N \to [0, 1]$ by letting $\phi_i=\phi'_i / \sum_{j\in\mathcal I} \phi'_j$, for each $i\in\mathcal I$. 
    
    Note that for each $i\in \mathcal I$, we have $\rho(\mathcal N, K_i) > 0$. Indeed, each $K_i \subseteq \mathcal M \setminus \mathcal N$ is non-empty, so $\rho(\mathcal N, K_i) > 0$ is trivially true if $K_i$ is a singleton. Otherwise, if $K_i$ contains at least two points, then $0<\diam K_i \leq d\cdot\rho(K_i, \mathcal N)$ as per \cref{it:whitney_3}.
    
    Consequently, we can pick $x_i\in \mathcal N$ such that $\rho(x_i, K_i)\leq 2\rho(\mathcal N, K_i)$. By \cref{it:whitney_3}, it follows that for any $i\in\mathcal I$ and $x\in K_i$,
    \begin{equation}\label{eq:dist_K}
        \rho(x, x_i)\leq \rho(x_i, K_i) + \operatorname{diam} K_i \leq (2+d)\cdot \rho(\mathcal N, K_i)\text.
    \end{equation}

    Consider a Lipschitz map $f: \mathcal N \to Z$ into some $p$-Banach space $Z$. We may, without loss of generality, assume that $\Lip f = 1$. 
    
    We define an extension $f' : \mathcal M \to Z$ as $f'(x) = f(x)$ if and only if $x\in \mathcal N$ and $f(x) = \sum_{i\in\mathcal I} \phi_i (x)f(x_i)$ otherwise. In the rest of the proof, we will estimate the Lipschitz constant of $f'$. Note that since the Lipschitz constant of $f'$, when restricted to $\mathcal N$, is the same as that of $f$, we can consider the following three cases.

    \begin{casesp}
        \item\label{cs:est_case1} Assume that $x\in \mathcal N$ and $y\in \mathcal M\setminus \mathcal N$.
        
        Note that for all $i\in\mathcal I$ with $\phi_i(y)>0$, we have $\rho(y, x_i)\leq (2+d)\cdot\rho(\mathcal N, K_i) \leq (2+d)\cdot\rho(x, y)$ by \eqref{eq:dist_K}. Consequently, $\rho(x, x_i) \leq \rho(x, y)+\rho(y, x_i) \leq (3+d)\rho(x, y)$, and 
        \begin{equation*}
            \begin{split}
                \left\lVert f'(x)-f'(y)\right\rVert_p&=\left\lVert\sum_{i\in\mathcal I} (f'(x)-f'(x_i))\phi_i(y)\right\rVert_p \\
                &\leq C(p, o) \max_{\substack{i\in\mathcal I\\ \phi_i(x) > 0}}\rho(x, x_i)\\
                &\leq C(p, o) (3+d)\cdot\rho(x, y) \text.
            \end{split}
        \end{equation*}
        
        \item\label{cs:est_case2} Let $x,\,y\in \mathcal M\setminus \mathcal N$ be such that $\phi_i(x)\phi_i(y)=0$ for all $i\in\mathcal I$.
        
        Recall that by \cref{it:whitney_2}, there exists $i\in\mathcal I$ with $\rho(x, \mathcal M\setminus K_i) \geq s\cdot\rho(x, \mathcal N)$, so that, in particular, $\phi_i(x)>0$. It follows that $\phi_i(y)=0$, i.e., $y \in \mathcal M\setminus K_i$, and $\rho(x, y) \geq \rho(x, \mathcal M\setminus K_i)\geq s\cdot\rho(x, \mathcal N)$.

        Pick $x'\in \mathcal N$ such that $\rho(x', x) \leq 2\rho(x, \mathcal N)$. By the previous paragraph, $\rho(x', x)\leq 2/s \cdot \rho(x, y)$ and $\rho(x', y) \leq \rho(x', x) +\rho(x, y) \leq (2/s +1)\cdot\rho(x, y)$. It follows from \Cref{cs:est_case1} that
        \begin{equation*}
            \begin{split}
                \left\lVert f'(x)-f'(y)\right\rVert_p &\leq \left(\left\lVert f'(x)-f'(x')\right\rVert^p_p + \left\lVert f'(x')-f'(y)\right\rVert^p_p\right)^{1/p} \\
                &\leq C(p, o) (3+d)\left(\rho^p(x, x') + \rho^p(x', y)\right)^{1/p}\\
                &\leq C(p, o) (3+d)\left(2^{1+p}/s^p +1\right)^{1/p}\cdot\rho(x, y)\text,
            \end{split}
        \end{equation*}
        where, for the third inequality, we used the fact that $(2/s +1)^p \leq 2^p/s^p +1$ by subadditivity of $x\mapsto x^p$.

        \item Finally, let $x,\, y\in \mathcal M\setminus \mathcal N$ and $i\in\mathcal I$ be such that $\{x, y\} \subseteq K_i$.
        
        Note that whenever $x\in K_j$ for some $j\in \mathcal I$, then according to \eqref{eq:dist_K}, we have
        \begin{equation}\label{eq:ineq_ij}
            \begin{split}
                \rho(x_i, x_j)&\leq \rho(x_i, x) +\rho(x, x_j)\\
                &\leq (2+d)(\rho(\mathcal N, K_i)+\rho(\mathcal N, K_j))\\
                &\leq 2(2+d)\cdot\rho(x, \mathcal N) \text,
            \end{split}
        \end{equation}
        and similarly for $y$. 
        
        We may rewrite $f'(x)-f'(y)=\sum_{j\in\mathcal I} (f(x_j)-f(x_i))(\phi_j(x)-\phi_j(y))$. In light of \eqref{eq:ineq_ij}, we will now estimate the sum $\sum_{j\in\mathcal I} |\phi_j(x)-\phi_j(y)|$.
    
        To that end, let $\mathcal I' = \{ j\in\mathcal I : \phi_j(x) > 0 \text{ or } \phi_j(y) >0\}$, where $|\mathcal I'| \leq 2o$ by \cref{it:whitney_1}. Consequently, we enumerate the set $\mathcal I'$ as $\mathcal I' = \{i_1, \ldots, i_n\}$ for some $n \leq 2o$. Also, we define $a=(a_k)_{k=1}^n, \, b=(b_k)_{k=1}^n\in\mathbb R^n$ by letting $a_k=\rho(x, \mathcal M\setminus K_{i_k})$, $b_k=\rho(y, \mathcal M\setminus K_{i_k})$ for each $k\in\{1, \ldots, n\}$.
    
        For every $k\in \{1, \ldots, n\}$, we define $d_k: [0, 1] \to \mathbb R$ as \[ \lambda \mapsto \frac{(\lambda a_k + (1-\lambda)b_k)^m}{\sum_{l=1}^n (\lambda a_l + (1-\lambda)b_l)^m}, \quad \lambda \in [0, 1] \text.\] 
        Note that $d_k(1)=\phi_{i_k}(x)$ and $d_k(0)=\phi_{i_k}(y)$, for each $k\in \{1, \ldots, n\}$.
        
        Pick $k\in\{1,\ldots,n\}$. A short computation shows that for any $\lambda\in (0, 1)$,
        \begin{equation*}
            \begin{split}
                d'_k (\lambda) &= m\frac{(\lambda a_k + (1-\lambda)b_k)^{m-1}}{\sum_{l=1}^n (\lambda a_l + (1-\lambda)b_l)^m}(a_k-b_k) \\
                &\phantom{=}- m(\lambda a_k + (1-\lambda)b_k)^m \sum_{l=1}^n \frac{(\lambda a_l + (1-\lambda)b_l)^{m-1}}{(\sum_{l=1}^n (\lambda a_l + (1-\lambda)b_l)^m)^2}(a_l-b_l) \text.
            \end{split}
        \end{equation*}
        
        Taking the absolute value and summing across $k$, we obtain 
        \begin{equation}\label{eq:dist_derivative}
            \sum_{k=1}^n |d'_k| (\lambda) \leq 2m \frac{\sum_{l=1}^n (\lambda a_l + (1-\lambda)b_l)^{m-1}}{\sum_{l=1}^n (\lambda a_l + (1-\lambda)b_l)^m}\cdot\rho(x, y), \quad \lambda\in (0, 1)\text,
        \end{equation}
        where we used the fact that $|a_l-b_l|=|\rho(x, \mathcal M\setminus K_{i_l})-\rho(y, \mathcal M\setminus K_{i_l})|\leq \rho(x, y)$ for any $l\in \{1,\ldots,n\}$.
        
        By \Cref{lemma:ineq_means} with $n\leq 2o$, we deduce for every $\lambda\in (0, 1)$,
        \begin{equation}\label{eq:dist_der_mean}
            \sum_{k=1}^n |d'_k| (\lambda) \leq 2m(2o)^{1/m}\left(\sum_{k=1}^n (\lambda a_k + (1-\lambda)b_k)^m\right)^{-1/m}\cdot \rho(x, y) \text.
        \end{equation}
        
        Also, recall that by \cref{it:whitney_2}, there exist $k',\,k''\in\{1,\ldots, n\}$ with $\rho(x, \mathcal M\setminus K_{i_{k'}}) \geq s\cdot\rho(x, \mathcal N)$ and $\rho(y, \mathcal M\setminus K_{i_{k''}}) \geq s\cdot\rho(y, \mathcal N)$. Consequently, we get for any $\lambda \in (0, 1)$,
        \begin{equation*}
            \begin{split}
                \sum_{k=1}^n (\lambda a_k + (1-\lambda)b_k)^m &\geq \sum_{k\in\{k', k''\}} (\lambda a_k + (1-\lambda)b_k)^m\\
                &\geq (\min \{a_{i_{k'}}, b_{i_{k''}}\}/2)^m \\
                &\geq (s/2)^m(\min \{\rho(x, \mathcal N), \rho(y, \mathcal N)\})^m \text.
            \end{split}
        \end{equation*}

        Hence, continuing~\eqref{eq:dist_der_mean}, we obtain
        \begin{equation}\label{eq:estimate_sum}
            \begin{split}
                \sum_{j\in\mathcal I} |\phi_j(x)-\phi_j(y)| &= \sum_{k=1}^n |d_k(1)-d_k(0)|\\
                &\leq \sum_{k=1}^n\int_0^1 |d'_k| (\lambda)\,d\lambda\\
                &\leq 4/s\cdot m(2o)^{1/m}\\
                &\quad\cdot \left(\min \{\rho(x, \mathcal N), \rho(y, \mathcal N)\}\right)^{-1} \cdot \rho(x, y) \text.
            \end{split}
        \end{equation}
        Consequently,
        \begin{equation*}
            \begin{split}
                \lVert f'(x)-f'(y) \rVert_p &= \left\lVert\sum_{j\in\mathcal I} (f(x_j)-f(x_i))(\phi_j(x)-\phi_j(y))\right\rVert_p\\
                &\leq C(p, 2o)\max_{j\in \mathcal I'} \lVert f(x_j)-f(x_i) \rVert_p \\
                &\quad\cdot\sum_{j\in\mathcal I'} |\phi_j(x)-\phi_j(y)|\\
                &\leq \rho(x,y) \cdot C(p, 2o) (8/s)(2+d)m(2o)^{1/m} \\
                &\quad\cdot \frac{\max \{\rho(x, \mathcal N), \rho(y, \mathcal N)\}}{\min \{\rho(x, \mathcal N), \rho(y, \mathcal N)\}} \text,
            \end{split}
        \end{equation*}
        where the third inequality follows from \eqref{eq:ineq_ij} and \eqref{eq:estimate_sum}.

        We note that $\max \{\rho(x, \mathcal N), \rho(y, \mathcal N)\}/ \min \{\rho(x, \mathcal N), \rho(y, \mathcal N)\}<a$ by \cref{it:whitney_4}. Taking $m=\log_2 (2o)$, we get that \[ m(2o)^{1/m}=\log_2 (2o) \cdot 2^{\log_2 (2o)/\log_2 (2o)}\leq 2\log_2 (2o) \text. \]
    \end{casesp}
    
    Altogether, it follows that $\Lip f' \leq D \cdot C(p, o) \log_2 (2o)$, where $D$ depends only on $p$, $s$, $d$, and $a$. Quantitatively, we may set 
    \begin{equation}\label{eq:const_D}
        D(p, s, d, a)=\max \left\{32/s\cdot C(p,2)(d+2)a, (3+d)\left(2^{p+1}/s^p +1\right)^{1/p}\right\} \text.
    \end{equation}

    The proof is complete.
\end{proof}

We remark that the proof actually shows a stronger claim. Specifically, it establishes the existence of \emph{simultaneous Lipschitz extension} from $\mathcal N$ to $\mathcal M$ (see \cite{Brudnyi2005}). That is, there is a bounded linear extension operator $\Lip(\mathcal N, Z) \to \Lip(\mathcal M, Z)$, wherein the spaces of Lipschitz maps are endowed with the Lipschitz semi-norm. Also note, by the method of proof, that $\operatorname{Rng} f' \subseteq \operatorname{conv} \operatorname{Rng} f$.

\section{Covers for Metrics with Finite Nagata Dimension}\label{sec:nagata_covers}

We recall that the concept of Nagata dimension is closely related to that of the asymptotic dimension (consider \textcite{Gromov1993,Bell2008}). The Nagata dimension was first considered by \textcite{Nagata1958}, and its contemporary formulation is due to \textcite{Assouad1982}. As we will show, doubling spaces and metric graphs provide examples of spaces with a finite Nagata dimension. Intriguingly, for a metric space $(X, \rho)$ with a Nagata dimension not exceeding $d$, its corresponding metric snowflake $(X, \rho^p)$ (for sufficiently small exponents $0<p<1$), can be embedded into a product of $d+1$ metric trees (see \cite[Theorem 1.3]{Lang2005}). Furthermore, we recall that ultrametric spaces have Nagata dimension zero with constant 1 (also see \cite{Basso2023}).

The aim of this section is to establish a general theorem affirming the existence of Whitney covers for spaces $\mathcal N$ with finite Nagata dimension. Subsequently, we examine particular distinguished classes of such spaces, along with estimates on their dimension.

The essence of the proof showing the existence of Whitney covers involves partitioning $\mathcal M \setminus \mathcal N$ into a collection of expanding metric annuli within $\mathcal M \setminus \mathcal N$. These annuli increase in size directly proportional to their distance from $\mathcal N$, and the corresponding covering sets of these annuli are then induced by Nagata covers of $\mathcal N$, associated with various $s>0$. In this context, \cref{it:nagata_3} controls the overlapping constant of the cover. The details follow.

\begin{prop}\label{prop:whitney_cover_nagata}
    Let $\mathcal N$ be a closed subset of a metric space $(\mathcal M, \rho)$. If $\mathcal N$ has Nagata dimension at most $d$ with constant $\lambda$, then there exists a Whitney cover of $\mathcal M \setminus \mathcal N$ with parameters $(3(d+1), \epsilon/2, 4(1+\gamma)(\epsilon+2), 5)$, for each $0<\epsilon<1/2$.
\end{prop}

\begin{proof}
    For each $j\in\mathbb Z$, we consider the metric annuli \[ A_j = \{ x\in \mathcal M\setminus \mathcal N: 2^j \leq \rho(x, \mathcal N) < 2^{j+1} \} \text. \] Also, for every $j\in \mathbb{Z}$ we fix a cover $\mathcal{C}^j$, associated with $s_j > 0$ (to be optimized later), that demonstrates that $\mathcal{N}$ has Nagata dimension at most $d$ with coefficient $\lambda$. We then set \[ B^j_C = \{ x\in A_j : \rho (x, \mathcal N) = \rho(x, C)\}\text, \quad \text{where $C\in\mathcal C^j$ and $j\in\mathbb Z$.} \]

    We assert that for any $j\in\mathbb Z$, if $s_j$ is such that $s_j \geq 2^{j+2}$, then $\{ B^j_C : C\in\mathcal C^j\}$ is a cover of $A_j$. For this, note that for any $x\in A_j$, the open ball $U(x, 2^{j+1})$ meets $\mathcal N$, and, moreover, $U(x, 2^{j+1}) \cap \mathcal N$ has diameter at most $2\cdot 2^{j+1}$. Consequently, $\{C\in\mathcal C^j : U(x, 2^{j+1}) \cap C \neq \emptyset \}$ is finite by \cref{it:nagata_3}. It follows that
    \begin{equation*}
        \begin{split}
            \rho (x, \mathcal N) &= \inf \{ \rho(x, y) : y\in \mathcal N\}\\
            &= \inf \{ \rho(x, y) : y\in U(x, 2^{j+1}) \cap \mathcal N\}\\
            &= \min \{ \rho(x, C) : C\in\{C\in\mathcal C^j : U(x, 2^{j+1}) \cap C \neq \emptyset \}\} \text,
        \end{split}
    \end{equation*}
    and the assertion is established.

    In the next step, we define a Whitney cover of $\mathcal M \setminus \mathcal N$ as follows. We pick $0<\epsilon<1/2$ and for each $j\in\mathbb Z$ and $C\in\mathcal C^j$, we set $K^j_C = \bigcup \{ U(x, \epsilon \cdot 2^j): x\in B^j_C \}$.

    Note that for any $x\in \mathcal M \setminus \mathcal N$ and $j\in\mathbb Z$, if $x\in K^j_C$ for some $C \in \mathcal C^j$, then $U(x, (\epsilon+2)\cdot 2^j) \cap C \neq \emptyset$. Consequently, provided that $s_j \geq 2^{j+1} (\epsilon+2)$, we get that $|\{ C \in\mathcal C^j: U(x, (\epsilon+2)\cdot 2^j) \cap C \neq \emptyset \}| \leq d+1$ by \cref{it:nagata_3}.

    Recall that $B^j_C \subseteq A_j$ for any $j\in\mathbb Z$ and $C\in\mathcal C^j$, so that $K^j_C \subseteq \{ x\in \mathcal M\setminus \mathcal N: 2^{j-1} \leq \rho(x, \mathcal N) < 5/2\cdot 2^j \}$. This establishes \cref{it:whitney_4} with $a=5$. Moreover, for every $j\in\mathbb Z$ we have 
    \begin{equation}\label{eq:K_intersection}
        \{ i\in\mathbb Z : A_j \cap K^i_C \neq\emptyset \text{ for some $C\in\mathcal C^i$}\} \subseteq \{ j-1, j, j+1\} \text.
    \end{equation}

    Together with the previous part, it follows that the overlapping constant $o$ from \cref{it:whitney_1} is at most $3(d+1)$, provided that $s_j \geq 2^{j+1} (\epsilon+2)$ for each $j\in\mathbb Z$.

    To adress \cref{it:whitney_2}, we assert that for any $x\in \mathcal M\setminus \mathcal N$, if $x\in B^j_C$ for some $j\in\mathbb Z$ and $C\in\mathcal C^j$, then $\rho(x, \mathcal M\setminus K^j_C) \geq \epsilon/2 \cdot \rho(x, \mathcal N)$. Indeed, note that $\rho(x, \mathcal M\setminus K^j_C) \geq \epsilon \cdot 2^j$ by the definition of $K^j_{\mathcal C}$, and $2^j \geq 1/2 \cdot \rho(x, \mathcal N)$ as $x\in A_j$.
    
    Finally, for \cref{it:whitney_3}, we estimate the diameter of $K^j_C$ for any given $j\in\mathbb Z$ and $C\in\mathcal C^j$. To that end, note that by the definition, there exists $y\in B^j_C$ such that $\rho(x, y) < \epsilon \cdot 2^j$, and, consequently, $\rho(x, C) \leq \rho(x, y) + \rho(y, C) < (\epsilon+2)\cdot 2^j$. It now follows that for any $x,\,x'\in K^j_C$, we have $\rho(x, x') \leq \rho(x, C) + \diam C + \rho(x', C) < 2(\epsilon+2)\cdot 2^j + \gamma s_j$. That is, $\diam K^j_C \leq (\epsilon+2)\cdot 2^{j+1} + \gamma s_j$.
    
    Choosing $s_j = 2^{j+1}(\epsilon+2)$ for each $j\in\mathbb Z$, we get $\diam K^j_C < 2^{j+1}(\epsilon+2)(1+ \gamma) \leq 4(1+\gamma)(\epsilon+2)\cdot \rho(\mathcal N, K^j_C)$. Here, we used that $2^{j-1} \leq \rho(\mathcal N, K^j_C)$ by the definition of $K^j_C$.
    
    It follows that the sets $K^j_C$, where $j\in\mathbb Z$ and $C\in\mathcal C^j$, form a Whitney cover of $\mathcal M\setminus \mathcal N$ with parameters $(3(d+1), \epsilon/2, 4(1+\gamma)(\epsilon+2), 5)$.
\end{proof}

Overall, in light of \Cref{thm:lipschitz_extension_p}, we have established the following.

\begin{thm}\label{thm:ae_p_nagata}
    If $\mathcal N$ has Nagata dimension at most $d$ with constant $\gamma$, then $\mathfrak{ae}_p (\mathcal N) \lesssim_p \gamma \cdot (d+1)^{1/p -1} \cdot \log (d+2)$ for all $0<p\leq 1$.
\end{thm}

A similar argument to the one used in the proof of \Cref{prop:whitney_cover_nagata} can be used to construct Whitney covers when $\mathcal N$ is doubling, as demonstrated, for example, in \cite[Proposition 5.2]{Albiac2021sums}. However, to the best of our knowledge, this approach would still provide an estimate that is quantitatively the same (in terms of $\lambda_{\mathcal N}$) as the one we obtain by estimating the Nagata dimension of doubling metric spaces, followed by applying \Cref{prop:whitney_cover_nagata}.

\begin{lemma}\label{lemma:doubling_nagata}
    Let $(\mathcal N, \rho)$ be a doubling metric space with doubling constant $\lambda_{\mathcal N}>1$. Then $\mathcal N$ has Nagata dimension at most $\lambda^3_{\mathcal N}-1$ with constant 2.
\end{lemma}

\begin{proof}
    Let $s>0$ be given. By the Zorn lemma, we obtain a maximal subset $S$ of $\mathcal N$ that is $s$-separated, i.e., $\rho(x, x') > s$ for every $x,\,x'\in S$.

    We claim that for any $x\in\mathcal N$, the closed ball $B(x, 3s)$ contains no more than $\lambda_{\mathcal N}^3$ elements of $S$. Given that $\mathcal N$ is doubling, we can cover $B(x, 3s)$ with $\lambda_{\mathcal N}^3$ closed balls of radius $s/2$. Each of these balls contains a maximum of one element of $S$, thus proving our claim.

    Consequently, by \cite[Lemma 2.4]{Assouad1983} applied to the subspace $S$, there is a coloring $k : S \to \{ 1, \ldots, \lambda_{\mathcal N}^3 \}$ such that $k(x) \neq k(x')$ whenever $\rho(x, x') \leq 3s$, for any $x,\,x'\in S$ with $x\neq x'$.

    It is easy to see that for any subset $A \subseteq \mathcal N$ with $\diam A \leq s$ and each $i\in\operatorname{Rng}{k}$, there exists at most one $x\in k^{-1} (i)$ for which the closed ball $B(x, s)$ intersects $A$.

    Therefore, if we denote $\mathcal C_i = \{ B(x, s) : x\in k^{-1} (i) \}$, where $i\in\operatorname{Rng}{k}$, and define $\mathcal C = \bigcup \{ \mathcal C_i : i\in\operatorname{Rng}{k}\}$, then $\mathcal C$ satisfies \cref{it:nagata_1,it:nagata_2,it:nagata_3} with $d=\lambda_{\mathcal N}^3 -1$ and $\lambda = 2$.
\end{proof}

Consequently, we obtain the following corollary.

\begin{cor}\label{cor:ae_p_doubling}
    If $\mathcal N$ has doubling constant $\lambda_{\mathcal N} > 1$, then $\mathfrak{ae}_p (\mathcal N) \lesssim_p \lambda_{\mathcal N}^{3/p-3} \cdot \log \lambda_{\mathcal N}$ for all $0<p\leq 1$.
\end{cor}

Continuing the applications of \Cref{thm:ae_p_nagata}, we note that an important class of spaces with finite Nagata dimension is provided by \emph{metric trees}, which are geodesic metric spaces where all geodesic triangles are degenerate. Given that every metric tree has Nagata dimension at most 1 with constant 6 (see \cite[Proposition 3.2]{Lang2005}, where the result is formulated without stating the exact value of the constant $6$; however, this value can be recovered from the proof), we derive the following generalization of a classical extension result due to \textcite{Matousek1990}. Note that we implicitly use the fact that if a space $\mathcal N$ has a Nagata dimension of at most $d$ with a constant $\gamma$, then any of its subspaces does as well.

\begin{cor}\label{cor:ae_p_trees}
    If $T$ is a metric tree and $S\subseteq T$ any its subset, then $\mathfrak{ae}_p (S) \lesssim_p 1$ for any $0<p\leq 1$.
\end{cor}

As we will see, \Cref{cor:ae_p_trees} is in fact a specific corollary of a more general result showing that a \emph{metric graph} induced by a \emph{weighted graph}, which excludes the complete graph $K_m$ as a \emph{minor}, has finite Nagata dimension. In particular, this dimension, along with the associated constant $\gamma$, can be bounded in terms of $m$. The connection to metric trees is that they are induced by weighted graphs which do not contain cycles, meaning they exclude $K_3$, the complete graph on three vertices, as a minor.

Before we generalize this result to graphs that exclude minors of higher degrees, we recall the following definition.

\begin{defn}[{\cite{Lee2004}}]\label{def:metric_graph}
    Let $G = (V, E)$ be a connected undirected graph and let $\phi : E \to [0, \infty)$ be a weight function that assigns weights to the edges in $E$. Consider the \emph{one-dimensional simplicial complex} $\Sigma (G)$, which is formed by replacing every edge $e\in E$ with a segment of length equal to $\phi(e)$. The \emph{metric graph $(\Sigma(G), \rho_\phi)$ induced by $G$ and $\phi$} denotes $\Sigma (G)$ with its inherent Riemannian semimetric structure, which restricts to the shortest path semimetric on the vertices of $G$.
\end{defn}

Note that even when all the weights are positive, if the set of edges is infinite, the resulting space $(\Sigma (G), \rho_\phi)$ can fail to satisfy the separation axiom. In other words, it becomes a \emph{semimetric} space. However, in what follows, we can readily adapt the metric notions to the semimetric setting, or implicitly transition to an induced quotient space $(\Sigma(G) / \sim, \rho'_\phi)$, where we set $v$ to $w$ if and only if $\rho_\phi (v, w) = 0$. There will be no loss of generality in the present context due to these considerations.

\begin{defn}
    Given a connected graph $G = (V, E)$, we say that a graph $G'$ is a \emph{minor of $G$} if it can be derived from $G$ in a finite sequence of the following operations:
    \begin{enumerate}[label=(\roman*), ref=\roman*]
        \item removing an edge $\{v, w \}\in E$, which results in a graph $G' = (V, E')$, where $E' =  E \setminus \{ e \}$,
        \item contracting an edge $\{v, w\} = e\in E$, which results in a graph $G' = (V', E')$, where $V' = V \setminus \{ w \}$ and $E' = E \setminus \{ e \in E: w\in e \} \cup \{ \{v, u\}: \{w, u\} \in E,\, u\neq v\}$.
    \end{enumerate}
\end{defn}

The concept of minor exclusion is of significant importance in the theory of topological and geometric properties of graphs. For example, it is closely linked to the problem about embeddability of graphs into Banach spaces. We recall that the Kuratowski Theorem characterizes planar graphs only in terms of minor exclusion (refer also to \cite{Tverberg1988}).

More relevant to the topic of this paper, there has also been considerable interest in the study of dimensionality invariants of graphs, focusing particularly on the asymptotic and Nagata dimensions of semimetrics induced by minor-excluded graphs and minor-closed families of graphs. Consider the contributions by \textcite{Gromov1993,Ostrovskii2015,Liu2023,Bonamy2023,Distel2023}, to name a few.

Specifically, \textcite{Ostrovskii2015} established that the Nagata dimension of \emph{unweighted graphs} that exclude the complete graph on $m$ vertices, $K_m$, is at most $4^m-1$ with a constant that depends only on $m$. In this paper, we generalize their result to the context of weighted graphs. That is, we show the following.

\begin{prop}\label{prop:nagata_minor_excluded}
    If $G=(V, E)$ is a countable and connected weighted graph which excludes the complete graph $K_m$ as a minor, then the metric graph $(\Sigma(G), \rho)$ has a Nagata dimension at most $3^{2m-2}-1$  with constant $\gamma \lesssim m$.
\end{prop}

Our approach is based on the general Klein--Plotkin--Rao decomposition technique for unweighted graphs, introduced by \textcite{Klein1993} and later refined by \textcite{Fakcharoenphol2003, Abraham2014}. This method also surfaced by \textcite{Ostrovskii2015}. Furthermore, it was applied to metric graphs by \textcite{Lee2004}, where the technique was employed to establish the existence of stochastic decompositions, as to address the Lipschitz extension problem.

The very general idea of the proof is to decompose the metric graph into a set of metric annuli and consider the connected components within these annuli. Then, we inductively apply the decomposition algorithm on each of these connected components.

\begin{notation}
    In what follows, if $C$ is a connected subset of $\Sigma (G)$, we let $\rho_C$ denote the \emph{weak distance} on $C$, i.e., $\rho_C(w, w')$ for $w,\,w' \in C$ is defined as the infimum over lengths of paths in $C$ connecting $w$ to $w'$.
\end{notation}

In the sequel, we will require the following technical result. Its proof is highly technical and relies on the decomposition technique borrowed from \textcite{Fakcharoenphol2003}, which has been adapted here to the setting of metric graphs. Consequently, the proof will be presented in the \nameref{sec:appendix}.

\begin{restatable}{lemma}{distance}\label{lemma:cluster_distances}
    Let $G$ be a countable and connected weighted graph such that there are $r>0$, $m\in \mathbb{N}$ with $m\geq 3$, $\delta\in \{0,1,2\}^{m-1}$, $\Sigma(G)= C_1\supseteq \ldots \supseteq C_m$, and points $s_i\in C_i$ for $i\in \{1,\ldots, m\}$ satisfying that for every $i,\,i'\in\{1, \ldots, m\}$, where $i<i'$, we have $\rho(s_i, s_{i'}) > 24mr$ and the set $C_{i+1}$ is a pathwise connected component of $A_{n_i} = \{v\in C_i\colon 3r(n_i-1)\leq \rho_{C_i}(v,s_i)-\delta_i r < 3rn_i\}$, for some $n_i\in\mathbb{N}_0$.
    
    Then $G$ contains $K_m$ as a minor.
\end{restatable}

We are now ready to give the proof of \Cref{prop:nagata_minor_excluded}.

\begin{proof}
    Let $r > 0$. In what follows, we can assume that $\diam \Sigma (G) > 24mr$, since otherwise $\{ \Sigma (G) \}$ would be a Nagata cover with the desired properties.

    We set $A_\emptyset = \{ \Sigma (G) \}$, $\mathcal I_\emptyset = \emptyset$, and $F_\emptyset (\Sigma (G)) = (v, )$ (which is a sequence of length one) for some $v\in \Sigma (G)$. For each $i\in \{ 1, \ldots, 2m-2\}$, we inductively define sets $A_{\delta},\,I_\delta \subseteq \mathcal P(\Sigma(G))$ and a map $F_\delta: A_{\delta} \to (\Sigma (G) \cup \{ \emptyset \})^{i+1}$ (without loss of generality, assume that $\emptyset\notin \Sigma (G)$), where $\delta \in \{ 0, 1, 2 \}^i$, with the following properties.
    
    Specifically, we will assume that for each $i\in\{0, \ldots, 2m-2\}$ and $\delta = (\delta_1, \ldots, \delta_i) \in \{ 0, 1, 2 \}^i$ (where we use the convention that $\delta = \emptyset$ if $i=0$), 
    \begin{enumerate}[label=(\roman*), ref=\roman*]
        \item\label{it:decomposition_partitions} $\mathcal A_\delta \cup \mathcal I_\delta$ is a partition of $\Sigma (G)$,
        \item\label{it:decomposition_monotonous} for any $C\in\mathcal A_\delta$, there exists a unique $C' \in \mathcal A_{(\delta_1, \ldots, \delta_{i-1})}$ such that $C \subseteq C'$. Moreover, $F_\delta (C) \supset F_{(\delta_1, \ldots, \delta_{i-1})} (C')$,
        \item\label{it:decomposition_I_diameter} for any $C\in \mathcal I_\delta$, we have $\diam C \leq (48m+6)r$,
        \item\label{it:decomposition_sequence} for any $C \in \mathcal A_\delta$, if we denote $F_\delta (C) = (s_j)_{j=1}^{i+1}$, then either $s_{i+1} \in C$ or $s_{i+1}=\emptyset$. Moreover, if $l$ is the greatest $l\in\{1,\ldots, i+1\}$ such that $s_{l'}\neq\emptyset$ for each $l'\leq l$, then $s_{l''} = \emptyset$ for every $l''>l$,
        \item\label{it:decomposition_sequence_distance1} if $C$, $(s_j)_{j=1}^{i+1}$, and $l$ are as above and, moreover, $l<i+1$, then for every $v\in C$ there exists $j\in\{1,\ldots, l\}$ such that $\rho(s_j, v) \leq 24rm$,
        \item\label{it:decomposition_A_empty} if $C$, $(s_j)_{j=1}^{i+1}$, and $l$ are as above, and if $j\in\{1,\ldots, l\}$ is such that $l+j\leq i$, then $\rho(s_j, v) > 24rm$ for any $v\in C$. Hence, by \cref{it:decomposition_sequence_distance1}, we have that $i < 2l$,
        \item\label{it:decomposition_sequence_distance2} if $C$, $(s_j)_{j=1}^{i+1}$, and $l$ are as above, then $\rho(s_j, s_{j'}) > 24rm$ for every $j,\,j' \leq l$, where $j\neq j'$,
        \item\label{it:decomposition_annuli} if $C$ and $(s_j)_{j=1}^{i+1}$ are as above, and if $s_i\neq\emptyset$ and $C' \in \mathcal A_{(\delta_1, \ldots, \delta_{i-1})}$ is such that $C\subseteq C'$, then $C$ is a pathwise connected component of $\{v\in C'\colon 3r(n-1)\leq \rho_{C'}(v,s_i)-\delta_{i} r < 3rn\}$ for some $n\in\mathbb{N}_0$.
    \end{enumerate}

    Moreover, we will assume that
    \begin{enumerate}[start=9,label=(\roman*), ref=\roman*]
        \item\label{it:decomposition_ball_inclusion} for any $v\in \Sigma (G)$, there exists $\delta\in\{ 0, 1, 2 \}^i$ and $C \in \mathcal A_\delta \cup \mathcal I_\delta$ such that the open ball of radius $r$ centered at $v$ is contained in $C$, i.e., $U(v, r) \subseteq C$.
    \end{enumerate}
    
    Let $i\in \{1, \ldots, 2m-2\}$ be such that for any $\delta \in \{ 0, 1, 2 \}^{i-1}$, the sets $\mathcal A_{\delta}$, $\mathcal I_{\delta}$, and the associated map $F_{\delta}$ were defined, satisfying \cref{it:decomposition_partitions,it:decomposition_I_diameter,it:decomposition_sequence,it:decomposition_sequence_distance1,it:decomposition_sequence_distance2,it:decomposition_A_empty}.

    Pick $\delta = (\delta_1, \ldots, \delta_{i-1}, \delta_i) \in \{ 0, 1, 2 \}^i$. If $\mathcal A_{(\delta_1, \ldots, \delta_{i-1})} = \emptyset$, we put $\mathcal A_\delta = \emptyset$ and $\mathcal I_\delta = \mathcal I_{(\delta_1, \ldots, \delta_{i-1})}$. Otherwise if $\mathcal A_{(\delta_1, \ldots, \delta_{i-1})}\neq\emptyset$, we define for each $C\in \mathcal A_{(\delta_1, \ldots, \delta_{i-1})}$ the sets $a_C$ and $i_C$ as follows. Let $(s_j)_{j=1}^i = F_{(\delta_1, \ldots, \delta_{i-1})} (C)$ and pick the greatest $l\in \{1, \ldots, i\}$ such that for any $l'\leq l$, we have $s_{l'} \neq\emptyset$. 
    
    Assume that $s_i \neq \emptyset$. We let $i_C = \emptyset$. As for $a_C$, we consider the annuli
    \begin{equation}\label{eq:metric_graph_annuli}
        A_n = \{v \in C: 3r(n - 1) \leq \rho_C(v, s_{i}) - \delta_i r < 3rn\}, \quad n\in\mathbb N_0 \text.
    \end{equation}
    Subsequently, if $A_n'$ denotes the set of disjoint pathwise connected topological components of $A_n$ for each $n \in \mathbb N_0$, we define $a_C = \bigcup_{n \in \mathbb N_0} A_n'$.

    If $s_i = \emptyset$, we consider the following three cases. Let $B = \{ v\in C: \rho(s_{i-l}, v) \leq 24mr \}$, where we use that $i-l\leq l$ by \cref{it:decomposition_A_empty}. If $B$ is empty, we put $a_C = \{ C \}$ and $i_C = \emptyset$. If $B$ is non-empty and $B_{3r} = \{ v\in C: \rho_C(B, v) \leq 3r \}$ contains all $C$, we put $a_C = \emptyset$ and $i_C = \{ C \}$.

    In the remaining case, when $B$ and $C \setminus B_{3r}$ are non-empty, we put $B_{\delta_i r} = \{ v\in C: \rho_C(B, v) \leq \delta_i r \}$ and let $i_C$ and $a_C$ be the sets of connected components in $B_{\delta_i r}$ and $C \setminus B_{\delta_i r}$, respectively.

    Having defined $a_C$ and $i_C$ for each $C\in\mathcal A_{(\delta_1, \ldots, \delta_{i-1})}$, we put \[\mathcal A_\delta = \bigcup \{ a_C : C\in \mathcal A_{(\delta_1, \ldots, \delta_{i-1})} \} \] and \[\mathcal I_\delta = \mathcal I_{(\delta_1, \ldots, \delta_{i-1})} \cup \bigcup \{ i_C : C\in \mathcal A_{(\delta_1, \ldots, \delta_{i-1})} \} \text. \]
    
    Note that by the induction hypothesis, it is easy to see that $\mathcal A_\delta \cup \mathcal I_\delta$ is a partition of $\Sigma (G)$, which establishes \cref{it:decomposition_partitions}. Also, the first part of \cref{it:decomposition_monotonous} clearly is satisfied.

    To address condition \cref{it:decomposition_I_diameter}, by the inductive assumption is suffices to pick $C\in i_{C'}$ for some $C'\in \mathcal A_{(\delta_1, \ldots, \delta_{i-1})}$. Note that for any $C \in \mathcal I_{(\delta_1, \ldots, \delta_{i-1})}$, we can find $C'\in \mathcal A_{(\delta_1, \ldots, \delta_{i-1})}$ such that $C \in i_{C'}$. Also, if $i$ and $l$ are as in the definition of $i_{C'}$, it is easy to see that, indeed, $\rho(v, s_{i-l}) \leq 24mr+3r$ for any $v\in C$. Consequently, it follows from the triangle inequality that $\diam C \leq (48m+6)r$.
    
    For each $C\in \mathcal A_\delta$, we define $F_\delta (C)$ as follows. First, we pick $C' \in\mathcal A_{(\delta_1, \ldots, \delta_{i-1})}$ such that $C' \supseteq C$, i.e., $C \in a_{C'}$. If we denote $(s'_j)_{j=1}^i = F_{(\delta_1, \ldots, \delta_{i-1})} (C')$, we let $F_\delta (C) = (s_j)_{j=1}^{i+1}$ be such that $s_j = s'_j$ for any $j\in \{ 1, \ldots, i\}$. Also, if $s_i\neq\emptyset$ and there exists a point $v\in C$ at distance greater than $24mr$ from $\{ s_j : j\in\{1, \ldots, i\}\}$, we let $s_{i+1} = v$. Otherwise we put $s_{i+1} = \emptyset$. It is easy to see that $F_\delta$ satisfies the remaining part of \cref{it:decomposition_monotonous} and \cref{it:decomposition_sequence,it:decomposition_sequence_distance1,it:decomposition_sequence_distance2}. Also \cref{it:decomposition_annuli} clearly follows from the inductive assumption and the construction.

    As for \cref{it:decomposition_A_empty}, let $C$, $(s_j)_{j=1}^{i+1}$, and $l$ be as in the statement. By the inductive hypothesis, we can assume that $j=i-l \geq 1$. But then, if $C' \in\mathcal A_{(\delta_1, \ldots, \delta_{i-1})}$ is such that $C' \supseteq C$, it follows from the construction that any $v\in C'$ with $\rho(s_j, v) \leq 24mr$ is now contained in some component within $\mathcal I_\delta$. In other words, $\rho(s_j, v) > 24mr$ for any $v\in C$.

    To establish \cref{it:decomposition_ball_inclusion}, pick $v\in \Sigma (G)$ and assume that for some $(\delta_1, \ldots, \delta_{i-1})\in\{ 0, 1, 2 \}^{i-1}$ and $C \in \mathcal A_{(\delta_1, \ldots, \delta_{i-1})} \cup \mathcal I_{(\delta_1, \ldots, \delta_{i-1})}$, it holds that $U(v, r) \subseteq C$. Assume that $F_\delta (C) = (s_j)_{j=1}^{i}$ satisfies that $s_{i}\neq\emptyset$. Then, if $\delta_i$ is such that $\rho_C (v, s_{i}) - \delta_i r \in [3r(n - 1) + r, 3rn - r)$ for some $n \in \mathbb N_0$, then also $U(v, r) \subseteq A_n = \{w \in C: 3r(n - 1) \leq \rho_C (w, s_{i}) - \delta_i r < 3rn\}$. It is easy to see that $U(v, r)$ remains within a single connected component in $A_n$. 
    
    As for the remaining cases, note that if $s_{i}=\emptyset$ and $B$ and $C \setminus B_{3r}$ are as in the part where we defined the corresponding sets $a_C$ and $i_C$, it suffices to consider the case when $B\neq\emptyset$ and $C \setminus B_{3r} \neq\emptyset$. If $\rho_C(B, v) \leq r$, we put $\delta_i = 2$ and then $U(v,r)$ clearly remains within a single connected component in $B_{\delta_i r} = \{ u\in C: \rho_C(B, u) \leq 2r \}$. Otherwise if $\rho_C(B, v) > r$, we let $\delta_i = 0$, in which case $U(v,r)$ remains within a single connected component in $C\setminus B_{\delta_i r} = \{ u\in C: \rho_C(B, u) > 0 \}$.
    
    In the following part, we will use the sets $\mathcal I_{\delta} \subseteq \mathcal P(\Sigma(G))$, where $\delta\in\{0, 1, 2\}^{2m-2}$, to construct the actual Nagata cover of $\Sigma(G)$.

    To that end, let us first verify that $\mathcal A_{\delta} = \emptyset$ for each $\delta=(\delta_1, \ldots, \delta_{2m-2})\in\{0, 1, 2\}^{2m-2}$. Indeed, let us assume for a contradiction that there exists $C \in \mathcal A_{(\delta_1, \ldots, \delta_{2m-2})}$. It follows from \cref{it:decomposition_annuli,it:decomposition_sequence_distance2} and \Cref{lemma:cluster_distances} that if we put $F_{(\delta_1, \ldots, \delta_{2m-2})} (C) = (s_j)_{j=1}^{2m-1}$, then $s_m = \emptyset$, i.e., $l<m$. Consequently, \cref{it:decomposition_monotonous,it:decomposition_A_empty} show that $2m-2 < 2l \leq 2(m-1)$, which is absurd.
    
    In particular, it follows that each set $\mathcal I_{\delta}$ is a partition of $\Sigma (G)$.

    Similarly, \cref{it:decomposition_I_diameter} shows that the diameter of each set in $\mathcal I_\delta$ is bounded by $(48m+6)r$.
    
    For every $\delta \in \{0, 1, 2\}^{2m-2}$ and each element $C \in \mathcal I_{\delta}$, we consider the set
    \begin{equation}\label{eq:C_interior}
        C' = \{x \in C: \rho(x, \bigcup \{ S : S \in \mathcal I_\delta \setminus \{C\} \}) \geq r\}\text,
    \end{equation}
    and put $\mathcal I'_{\delta} = \{C': C \in \mathcal I_{\delta}\}$.

    It is easy to see from \eqref{eq:C_interior} that for any subset $A \subseteq \Sigma(G)$ with $\diam A \leq r/2$, we have
    \begin{equation*}
        |\{C' \in \mathcal I'_{\delta}: C' \cap A \neq \emptyset\}| \leq 1, \quad \delta \in \{0, 1, 2\}^{2m-2}\text.
    \end{equation*}
    Consequently,
    \begin{equation}\label{eq:intersection}
        |\{C' \in \bigcup \{ \mathcal I'_{\delta} : \delta \in \{0, 1, 2\}^{2m-2}\} : C' \cap A \neq \emptyset\}| \leq 3^{2m-2} \text.
    \end{equation}

    At the same time, for each $v\in \Sigma(G)$, if $\delta \in \{0, 1, 2\}^{2m-2}$ is such that $U(v, r) \subseteq C$ for some $C\in\mathcal I_\delta$ (see \cref{it:decomposition_ball_inclusion}), then also $v\in C'$. This shows that $\bigcup_{\delta \in \{0, 1, 2\}^{2m-2}} \mathcal I'_{\delta}$ is a cover of $\Sigma (G)$.

    Therefore, for any $r > 0$, we have constructed a cover $\mathcal C$ which satisfies \cref{it:nagata_1,it:nagata_2,it:nagata_3} with $s=r/2$, $d = 3^{2m-2} - 1$ (see \eqref{eq:intersection}), and $\gamma \lesssim m$, thus proving the claim.
\end{proof}

As a consequence, we obtain the following extension result for subset of metric graphs.

\begin{thm}\label{thm:ae_p_graphs}
    If $\Sigma(G)$ is a metric graph induced by a countable and connected weighted graph $G$ which excludes the complete graph $K_m$ as a minor, then $\mathfrak{ae}_p (S) \lesssim_p m^2 \cdot 9^{m(1/p - 1)}$ for any subset $S\subseteq \Sigma(G)$ and any $0<p\leq 1$.
\end{thm}

\section{Lipschitz Extensions and Lipschitz Free Spaces}\label{sec:lip_free_spaces}

Recall that in the definition of $p$-trace and of absolute $p$-extendability, we considered the existence of Lipschitz extensions for each Lipschitz map $f : \mathcal N \to Z$ that ranges into any $p$-Banach space $Z$. Nevertheless, we can develop these notions for a single, canonical choice of $Z$ along with the map $f$. This is facilitated by a universal object associated with $\mathcal N$, referred to as the \emph{Lipschitz free $p$-space over $\mathcal N$}.

\begin{thm}[{cf.~\cite[Theorem~4.5]{Albiac2020}}]
    Let $(\mathcal N, \rho)$ be a~pointed metric space. Given $0<p\leq 1$, there exists a $p$-Banach space $(\mathcal F_p (\mathcal N), \lVert\cdot\rVert)$, called the~\emph{Lipschitz free $p$-space over $\mathcal N$}, and a~map $\delta_{\mathcal N} : \mathcal N \to \mathcal F_p (\mathcal N)$ such that
    \begin{props}[label=(\roman*), ref=\roman*]
        \item\label{it:F_p_isometry} $\delta_{\mathcal N}$ is an~isometric embedding with $\delta_{\mathcal N}(0_{\mathcal N})=0_{\mathcal F_p(\mathcal N)}$,
        \item\label{it:F_p_dense} $\mathcal F_p (\mathcal N) = \overline{\operatorname{span}} \{ \delta_{\mathcal N}(x) : x\in \mathcal N \}$,
        \item\label{it:F_p_extension} for any Lipschitz map $f\in\Lip_0 (\mathcal N, Y)$, where $Y$ is a $p$-Banach space, there is a unique bounded linear operator $L_f : \mathcal F_p(\mathcal N) \to Y$ such that $L_f \circ \delta_{\mathcal N} = f$. This operator is called the \emph{canonical linearization of $f$}. Moreover, it satisfies that $\lVert L_f \rVert = \Lip f$.
    \end{props}
\end{thm}

A distinctive feature of Lipschitz free spaces is that they relate the classical linear theory to the non-linear geometry of Banach spaces. This was originally observed in the seminal paper by \textcite{Godefroy2003}, and subsequently pursued in works including \cite{Albiac2022, Borel2012, Kalton2012}, to name a few.

In particular, the problem about the $p$-trace of $\mathcal N$ in $\mathcal M$, a nonlinear phenomenon, is equivalent to the problem about existence of a specific linear operator between the Lipschitz free $p$-spaces $\mathcal F_p (\mathcal M)$ and $\mathcal F_p(\mathcal N)$. We note that the connection between the Lipschitz extension problem and the theory of Lipschitz free spaces was recently explored by \textcite{Albiac2021sums}.

\begin{defn}[{cf.~\cite[Definition 2.6]{Albiac2021sums}}]\label{def:cpa}
    If $\mathcal N$ is a subspace of a pointed metric space $\mathcal M$ and if $0<p\leq 1$, we say $\mathcal N$ is \emph{complementably $p$-amenable in $\mathcal M$ with constant $C<\infty$} provided there exists a bounded operator $T:\mathcal F_p (\mathcal M)\to \mathcal F_p (\mathcal N)$ satisfying $T\circ L_i= \operatorname{Id}_{\mathcal F_p (\mathcal N)}$ and $\lVert T\rVert \leq C$. Here, $L_i$ is the canonical linearization of the inclusion $i:\mathcal N \to \mathcal M$. In other words, $L_i$ is the unique linear map $L_i: \mathcal F_p (\mathcal N) \to \mathcal F_p (\mathcal M)$ such that for each $x\in\mathcal N$, we have $L_i(\delta_{\mathcal N}(x)) = \delta_{\mathcal M} (x)$.
\end{defn}

The exact relation between the constant of complementable $p$-amenability and the $p$-trace of $\mathcal N$ in $\mathcal M$ is described below.

\begin{prop}\label{prop:cpa_characterization}
    Let $\mathcal N$ be a subspace of a pointed metric space $\mathcal M$. Then, for each $0<p\leq 1$ and $0<C < \infty$, the following are equivalent:
    \begin{enumerate}[label=(\roman*), ref=\roman*]
        \item\label{it:cpa_1} $\mathcal N$ is complementably $p$-amenable in $\mathcal M$ with constant less than $C$,
        \item\label{it:cpa_2} the inclusion map $i: \mathcal N \to \mathcal F_p (\mathcal N)$ extends to a map $i': \mathcal M \to \mathcal F_p (\mathcal N)$ such that $\Lip i' < C$,
        \item\label{it:cpa_3} the $p$-trace of $\mathcal N$ in $\mathcal M$ is less than $C$, i.e., $\mathfrak{t}_p (\mathcal N, \mathcal M) < C$.
    \end{enumerate}
\end{prop}

\begin{proof}
    In order to prove that \cref{it:cpa_1} implies \cref{it:cpa_2}, observe that if $T:\mathcal F_p (\mathcal M)\to\mathcal F_p(\mathcal N)$ is the bounded operator from the definition of complementable $p$-amenability with constant $0<C'<C$, we can define $i'=T\circ \delta_{\mathcal M}$, where $\Lip i' \leq \lVert T\rVert\leq C' < C$ because $\delta_{\mathcal M}$ is an isometry. 

    Also, \cref{it:cpa_2} implies \cref{it:cpa_3}, as the canonical linearization $L_{i'} : \mathcal F_p (\mathcal M) \to \mathcal F_p (\mathcal N)$ of $i'$, given by \cref{it:F_p_extension} of the Lipschitz free $p$-space $\mathcal F_p (\mathcal M)$, satisfies $L_{i'} \circ L_i = \operatorname{Id}_{\mathcal F_p (\mathcal N)}$ with $\lVert L_{i'} \rVert < C$.

    To show that \cref{it:cpa_2} implies \cref{it:cpa_3}, note that for any $p$-Banach space $Z$ and a Lipschitz map $f: \mathcal N \to Z$, if $L_f : \mathcal F_p (\mathcal N) \to Z$ is the canonical linearization of $f$ given by \cref{it:F_p_extension} of the Lipschitz free $p$-space $\mathcal F_p (\mathcal N)$, we may take $f' = L_f \circ i'$. From this, it follows that $\Lip f' \leq \Lip i' \cdot \Lip f$ and, consequently, $\mathfrak{t}_p (\mathcal N, \mathcal M) \leq \Lip i' < C$. 

    Moreover, \cref{it:cpa_3} clearly implies \cref{it:cpa_2}, thereby proving the claim.
\end{proof}

Let us remark that in \cite{Albiac2021sums}, numerous results, related solely to the structural properties of Lipschitz free $p$-spaces over doubling metrics, were derived from the absolute $p$-extendability of this class of metrics. Having generalized these extendability results to the class of spaces with finite Nagata dimension in \Cref{sec:nagata_covers}, we can restate some of these structural results accordingly. For instance, consider the following.

\begin{cor}[{\emph{for spaces with finite Nagata dimension}; cf.~\cite[Corollary 5.3]{Albiac2021sums}}]
    If $\mathcal M$ is a metric space with finite \emph{Nagata dimension}, then there exists a net $(T_i)_{i\in\mathcal I}$ of finite-rank projections on $\mathcal F_p (\mathcal M)$, uniformly bounded in norm, that converges uniformly to the identity map on compact sets. In particular, $\mathcal F_p (\mathcal M)$ has the \emph{$\pi$-property}.
\end{cor}

We note that the same proof as in \cite[Corollary 5.3]{Albiac2021sums} is applicable here. For the proof of the following corollary, in contrast to its counterpart in doubling spaces, we additionally assume that $\mathcal M$ is a \emph{proper metric space}. That is, all closed, bounded subspaces of $\mathcal M$ are compact. Observe that not all countable spaces of finite Nagata dimension are proper. For example, any countably infinite set $\mathcal M$ equipped with the discrete metric $\rho(x, y) = 1$ if, and only if, $x,\,y\in\mathcal M$ and $x\neq y$, is not proper, although its Nagata dimension is 0.

\begin{cor}[{\emph{for spaces with finite Nagata dimension}; cf.~\cite[Corollary 5.4]{Albiac2021sums}}]
    If $\mathcal M$ is a complete countable \emph{proper} metric space with finite Nagata dimension, then $\mathcal F_1(\mathcal M)$ has the \emph{finite dimensional decomposition property}. In particular, it has the \emph{metric approximation property}, after suitable renorming (for the definition of approximation properties, see \cite{Casazza2001}).
\end{cor}

Lastly, we would like to highlight the following important application of the absolute $p$-extendability result for subsets of metric trees (refer to \Cref{cor:ae_p_trees}).

\begin{thm}[{cf.~\cite[Theorem 3.21]{Cuth2023}}]
    If $\mathcal N \subset \mathcal M$ are metric spaces in inclusion and $0<p\leq 1$, then the canonical linearization $T_i : \mathcal F_p (\mathcal N) \to \mathcal F_p (\mathcal M)$ of the inclusion $i: \mathcal N \to \mathcal M$ is an isomorphism. Moreover, $\lVert T^{-1}_i \rVert \lesssim_p 1$.
\end{thm}

To conclude this section, we show that the $p$-trace of $\mathcal N$ in $\mathcal M$ generally increases as $p$ approaches zero. Based on the results outlined above, it will suffice to consider extensions of the canonical inclusion $i: \mathcal N \to \mathcal F_p (\mathcal N)$.

\begin{thm}\label{thm:t_p_counterexample}
    Let $\mathcal N = \{0, 1, 2\} \subseteq (\mathbb R, |\cdot|)$ and $\mathcal M = \mathcal N \cup \{ 3/2 \}$. Then $\mathfrak{t}_1 (\mathcal N, \mathcal M) = 1$ but $\mathfrak{t}_p (\mathcal N, \mathcal M) > 1$ for any $0<p<1$. Moreover, we have $\mathfrak{t}_p (\mathcal N, \mathcal M) \to 2$ as $p\to 0$.
\end{thm}

In the proof, we will frequently need to estimate the $p$-norm of an element $m\in\mathcal F_p (\mathcal N)$. Let us note that while \cite[Theorem 2.2]{Cuth2023} provides a finite algorithm for the computation of the $p$-norm in Lipschitz free spaces over finite metrics, expressing the norm explicitly can still pose a significant challenge. However, when applied to the three-point space $\mathcal N$ under consideration, it yields the following formula.

\begin{fact}[{cf.~\cite[Corollary 2.7]{Cuth2023}}]\label{fact:p_norm}
    For any $0<p\leq 1$ and $x,\,y\in \mathbb R$, we have \[\lVert x\delta(1) + y\delta(2) \rVert_{\mathcal F_p (\mathcal N)}^p = \min \{ |x|^p + 2^p|y|^p, 2^p|x+y|^p + |x|^p, |x+y|^p + |y|^p\} \text. \]
\end{fact}

We are now ready to give the proof of the theorem.

\begin{proof}
    In line with \Cref{prop:cpa_characterization}, we will examine extensions $i' : \mathcal M \to \mathcal F_p (\mathcal N)$ of the canonical embedding $i : \mathcal N \to \mathcal F_p (\mathcal N)$. To this end, we develop estimates on the Lipschitz constant $\Lip i'$, with respect to the image of $3/2$ under this map.

    Observe that if $a,\,b\in\mathbb R$ are given and $i_{(a, b)}: \mathcal M \to \mathcal F_p (\mathcal N)$ is an extension of the canonical isometric embedding $i$, defined by $i_{(a, b)}(3/2) = a\delta(1) + b\delta(2)$, then the Lipschitz constant of $i_{(a, b)}$ equals
    \begin{multline}\label{eq:counterxample_l_constant}
        \Lip i_{(a, b)} = \max \{ 1, 2/3\cdot\lVert a\delta(1) + b\delta(2) \rVert_{\mathcal F_p (\mathcal N)}, \\ 2\lVert (a-1)\delta(1) + b\delta(2) \rVert_{\mathcal F_p (\mathcal N)}, 2\lVert a\delta(1) + (b-1)\delta(2) \rVert_{\mathcal F_p (\mathcal N)} \} \text.
    \end{multline}

    The $p$-norm of $\lVert a\delta(1) + b\delta(2) \rVert_{\mathcal F_p (\mathcal N)}$, as shown in \Cref{fact:p_norm}, is given by
    \begin{equation}\label{eq:counterexample_norm_expression}
        \lVert a\delta(1) + b\delta(2) \rVert_{\mathcal F_p (\mathcal N)}^p = \min \{ |a|^p + 2^p|b|^p, 2^p|a+b|^p + |a|^p, |a+b|^p + |b|^p\}\text.
    \end{equation}

    Similarly, the expression for $\lVert a\delta(1) + (b-1)\delta(2) \rVert_{\mathcal F_p (\mathcal N)}^p$ is
    \begin{equation}\label{eq:counterexample_p_norm}
        \min \{ |a|^p + 2^p|b-1|^p, 2^p|a+b-1|^p + |a|^p, |a+b-1|^p + |b-1|^p\} \text.
    \end{equation}
    
    Specifically, we deduce that $\lVert a\delta(1) + (b-1)\delta(2) \rVert_{\mathcal F_p (\mathcal N)}$ is bounded from below by $|1-b|$. Likewise, the value of $\lVert (a-1)\delta(1) + b\delta(2) \rVert_{\mathcal F_p (\mathcal N)}$ has a lower bound of $|1-a|$. As a result, we establish that
    \begin{equation}\label{eq:counterexample_lip_1}
        \Lip i_{(a, b)} \geq 2 \max \{ |1-b|,  |1-a|\} \text.
    \end{equation}
    In particular, this implies that $\Lip i_{(a, b)} > 2$, unless both $0\leq a,\,b\leq 2$.

    Note that $\Lip i_{(1, 0)} = 2$ for any $0<p<1$. This is evident upon noting that, by \cref{it:F_p_isometry} of the Lipschitz free $p$-space $\mathcal F_p (\mathcal N)$, we have $\lVert i_{(1, 0)}(3/2) - i_{(1, 0)}(x)\rVert_{\mathcal F_p (\mathcal N)} = \lVert \delta(1) - \delta(x) \rVert_{\mathcal F_p (\mathcal N)} =1$ for all $x\in \mathcal N \setminus \{ 1 \}$. Therefore, we deduce that $\mathfrak{t}_p (\mathcal N, \mathcal M) \leq \Lip i_{(1, 0)} = 2$.

    It follows that for any $0<p\leq 1$, there exists an extension $i_{(a^*, b^*)}$ with the smallest Lipschitz constant, which satisfies $0 \leq a^*,\,b^* \leq 2$. Indeed, this extension emerges as a minimum of \eqref{eq:counterxample_l_constant}, a continuous expression by \Cref{fact:p_norm}, over a compact interval $\{ (a, b): 0 \leq a,\, b \leq 2 \}$. Furthermore, given that the minimum is attained, \Cref{prop:cpa_characterization} yields $\mathfrak{t}_p (\mathcal N, \mathcal M)=\Lip i_{(a^*, b^*)}$.

    We assert that $\Lip i_{(a^*, b^*)} > 1$ for any given $0<p<1$. For the sake of contradiction, let us assume that $\Lip i_{(a^*, b^*)}=1$. Once again, referring to \eqref{eq:counterexample_lip_1}, we deduce that $\min \{ a^*, b^* \} \geq 1/2$. But then, \eqref{eq:counterexample_norm_expression} implies that $\lVert a^*\delta(1) + b^*\delta(2) \rVert_{\mathcal F_p (\mathcal N)} \geq (2^p+1)^{1/p}/2 > 3/2$, and, hence, $\Lip i_{(a^*, b^*)} > 1$ by \eqref{eq:counterxample_l_constant}, a contradiction.

    Lastly, we wish to show that $\mathfrak{t}_p (\mathcal N, \mathcal M) \to 2$ as $p\to 0$. To that end, we relabel the coefficients $a^*$ and $b^*$, associated with the optimal extensions, as $a_p$ and $b_p$ for each $0<p<1$.

    We assert that $\min \{ a_p, b_p \} \to 0$ as $p\to 0$. To prove this, suppose, on the contrary, that there exists $\epsilon > 0$ and a sequence $(p_i)_{i\in\mathbb N} \in (0, 1]^{\mathbb N}$, satisfying $p_i \to 0$ as $i\to\infty$, such that $\min \{ a_{p_i}, b_{p_i} \} > \epsilon$ for all $i\in\mathbb N$. However, then \eqref{eq:counterexample_norm_expression} reveals that $\lVert a_{p_i}\delta(1) + b_{p_i}\delta(2) \rVert_{p_i} > \epsilon (2^{p_i} + 1)^{1/p_i}$, with the right-side converging to infinity as $i\to\infty$. Consequently, by \eqref{eq:counterxample_l_constant}, we deduce that $\Lip i_{(a_{p_i}, b_{p_i})} \to \infty$ as $i\to\infty$. However, by our choice of $a_p$ and $b_p$, we also know that $\Lip i_{(a_{p_i}, b_{p_i})} \leq 2$ for any $i\in\mathbb N$, a contradiction.

    Altogether, having established that $\min \{ a_p, b_p \} \to 0$ as $p \to 0$, we note that \eqref{eq:counterexample_lip_1} implies that $\Lip i_{(a_p, b_p)} \to 2$ as $p\to 0$. This proves the final part, and the claim follows.
\end{proof}

\section{Open Problems}\label{sec:open_problems}

In \Cref{sec:nagata_covers}, we observed many times that the estimates on the absolute $p$-extendability constant increase as $p$ approaches zero, and specifically, they grow exponentially in terms of $1/p$. Subsequently, in \Cref{sec:lip_free_spaces}, we established a basic counterexample demonstrating that the $p$-trace typically does indeed depend on $p$. Consequently, we are curious to see how this behavior is exhibited across the different classes of spaces that we have considered in this paper, and whether we can derive lower estimates on this. 

In the following discussion, whenever $\mathcal N$ is a metric space with $\mathfrak{ae}_1 (\mathcal N) < \infty$, we denote $\mathfrak{q} (\mathcal N, p) = \mathfrak{ae}_p (\mathcal N)/ \mathfrak{ae}_1 (\mathcal N)$ for each $0<p\leq 1$. It is easy to see that $\lim_{p\to 0} \mathfrak{q} (\mathcal N, p)$ exists as $p\mapsto \mathfrak{ae}_p (\mathcal N)$ does not increase in $p$.

\begin{question}\label{question:ae_p_doubling}
    Is it true that for each $n\geq 2$, we have $\sup \{ \lim_{p\to 0} \mathfrak{q} (\mathcal N, p) : \mathcal N \text{ is a doubling with } \lambda_{\mathcal N} \leq n\} = \infty$? More specifically, can it be shown that for each fixed $n\geq 2$, $p\mapsto \log \sup \{ \mathfrak{q} (\mathcal N, p) : \lambda_{\mathcal N} \leq n\}$ grows proportionally to $1/p$ as $p\to 0$? What about if we consider $\mathcal N$ to only have a finite Nagata dimension at most $d$ with constant $\gamma$?
\end{question}

In relation to the example of \Cref{thm:t_p_counterexample}, we would like to know whether, if the answer to the very first question in \Cref{question:ae_p_doubling} is positive, it could potentially be identified within examples involving finite metric spaces.

Recall that by \Cref{cor:ae_p_doubling}, we have $\mathfrak{ae}_p (n) < \infty$ for every $n\in\mathbb N$ and $0<p\leq 1$ (we have a trivial estimate $\lambda_{\mathcal N} \leq n$ for every $n$-point metric space $\mathcal N$). Interestingly, observe also that in the proof of \Cref{thm:t_p_counterexample}, while the optimal projection was identified as a weighted sum of two points in the target Lipschitz free $p$-space $\mathcal F_p (\mathcal N)$ for $p=1$, the norm of that particular extension would become excessively large as $p\to 0$. In particular, as $p\to 0$, the Lipschitz norm of the optimal extension converged to that of a "trivial" extension that projected the additional point to a single point in $\mathcal F_p (\mathcal N)$. It is clear that such a projection does not require the underlying linear structure of the target space.

At the same time, note that a remark to \cite[Theorem 1.1]{Basso2018} shows that for each $m\in\mathbb N$, there exist spaces $\mathcal N \subset \mathcal M$ with $|\mathcal N| = 2$ and $|\mathcal M \setminus \mathcal N| = m$, such that if $f$ is the identity map $f: \mathcal N \to \mathcal N$, then $\Lip f' \geq (m+1)\Lip f$ for any extension $f':\mathcal M \to \mathcal N$ of $f$. Consequently, we pose the following question.

\begin{question}
    Is it true that $\lim_{p\to 0} \mathfrak{ae}_p (n) = \infty$ for each $n\geq 2$? Additionally, for any $n,\,m\in\mathbb N$, is it true that $\sup \{\lim_{p\to 0} \mathfrak{t}_p (\mathcal N, \mathcal M) : \mathcal N \subset \mathcal M \text{ with } |\mathcal N| \leq n \text{ and } |\mathcal M \setminus \mathcal N|\leq m \} = m+1$? What if we relax the condition that $|\mathcal N| \leq n$?
\end{question}

\begin{rem}\label{rem:basso}
    We wish to acknowledge that a generalization of the extension theorem from doubling spaces to spaces with finite Nagata dimension has been recently independently discovered by \textcite{Basso2023lipschitz}, in the narrower context of the Banach setting with $p=1$. This is the content of \Cref{introthm:ae_p_doubling} and, specifically, \Cref{thm:ae_p_nagata}. Interestingly, both methods of proof share similarities. Our result is more broadly applicable, dealing with $p$-Banach spaces for $p<1$. Conversely, additional classes of spaces are considered in \cite{Basso2023lipschitz}, specifically Lipschitz $n$-connected spaces (refer to \textcite{Lang2005}), and the author provides estimates on the absolute 1-extendability constant for finite metric spaces.
\end{rem}

\section*{Appendix}\label{sec:appendix}\appendix

We have adapted the results of \textcite{Fakcharoenphol2003} to the setting of metric graphs. It is important to note that, unlike in \cite{Fakcharoenphol2003} where the authors deal with the unweighted graph $G$, we cannot generally assume the existence of shortest-length paths in $\Sigma(G)$ and hence have to approximate these accordingly. Thus, if elements $u,\,v \subseteq C$ of some connected subset $C$ of $\Sigma(G)$ and $\epsilon > 0$ are given, we define a $p: [0, 1] \to C \subseteq \Sigma (G)$ to be an $\epsilon$-path in $C$ from $u$ to $v$ if $p$ is continuous and injective, $p(0)=u$, $p(1)=v$, and its length is less than $\rho_C(v, w) + \epsilon$. It is evident that every path can be written as a union of subpaths contained within the edges of $G$. Moreover, in what follows, it will be sufficient to fix any $0<\epsilon < r$.

\distance*

\begin{proof}
    We wish to construct for each $i\in \{2, \ldots, m-1\}$ sets $\mathcal{A}_{m-i+1}, \ldots, \mathcal{A}_{m}$ (called \emph{supernodes}) such that for each $j,\,j' \in \{m-i+1, \ldots, m\}$, where $j\neq j'$,
    \begin{enumerate}[label=(\roman*), ref=\roman*]
        \item\label{it:minor_1} $s_j \in \mathcal A_j \subseteq C_j$,
        \item\label{it:minor_2} the intersection of $\mathcal A_j$ and $\mathcal A_{j'}$ is a singleton,
        \item\label{it:minor_3} $\mathcal A_j \cap \mathcal A_{j'} \cap \mathcal A_{j''} = \emptyset$ for any given $j''\in \{m-i+1, \ldots, m\}\setminus \{j, j'\}$,
        \item\label{it:minor_9} if $T$ is the set of all points $x$ such that $x\in \mathcal A_l \cap\mathcal A_{l'}$ for some $l,\,l'\in \{m-i+1, \ldots, m\}$, where $l\neq l'$, then $\mathcal A_j \setminus T$ is pathwise connected.
    \end{enumerate}

    In addition, for each $j\in \{m-i+1, \ldots, m\}$, we construct a non-self-intersecting $\epsilon$-path $P_j$ connecting $u_j\in\mathcal A_j$ to $s_{m-i}$, a subpath $T_j \subset P_j$ (called a \emph{tail}) of length $24r$ connecting $u_j$ to a point $t_j$ (called the \emph{tip} of $T_j$) and a point $h_j \in T_j$ (called the \emph{middle point} of $T_j$), which is at distance $12r$ from $u_j$  along the path $P_j$.

    We will assume that for every $j,\,j' \in \{m-i+1, \ldots, m\}$, where $j\neq j'$,
    \begin{enumerate}[start=5, label=(\roman*), ref=\roman*]
        \item\label{it:minor_4} $T_j \cap (\mathcal A_{j'} \cup T_{j'}) = \emptyset$,
        \item\label{it:minor_5} the distance between the tails $T_j$ and $T_{j'}$ in $\Sigma (G)$ is greater than $24r(m-i)$. Moreover, $\rho(h_j, h_{j'}) > 24r(m-i+1)$,
        \item\label{it:minor_6} $\rho (h_j, s_{m-i}) > 24r(m-i+1)$,
        \item\label{it:minor_8} $\rho(s_j, h_j) \leq 12r(i-1)$.
    \end{enumerate}

    Moreover, we claim that the following properties are then always satisfied.
    
    \begin{claim}\label{claim:minor_properties}
        For every $j,\,j' \in \{m-i+1, \ldots, m\}$, where $j\neq j'$, it follows from \cref{it:minor_1,it:minor_2,it:minor_3,it:minor_4,it:minor_5,it:minor_6,it:minor_8,it:minor_9} that
        \begin{enumerate}[start=9, label=(\roman*), ref=\roman*]
            \item\label{it:minor_4'} $((P_{j} \setminus T_{j}) \cup \{t_{j}\}) \cap T_{j'} = \emptyset$, and $\mathcal A_j \cap ((P_{j} \setminus T_{j}) \cup \{t_{j}\})= \emptyset$,
            \item\label{it:minor_7} for any $w\in C_{m-i+1}$, we have $\rho_{C_{m-i}}(s_{m-i}, w) - \rho_{C_{m-i}}(s_{m-i}, h_{j})> 9r-\epsilon$ and $\rho_{C_{m-i}}(s_{m-i}, w) > 21r-\epsilon$. Also, for every $w'\in ((P_{l} \setminus T_{l}) \cup \{t_{l}\})$, where $l \in \{m-i+1, \ldots, m\}$, we have $\rho_{C_{m-i}}(s_{m-i}, h_{j}) - \rho_{C_{m-i}}(s_{m-i}, w') > 6r - \epsilon$.
        \end{enumerate}
    \end{claim}
    
    \begin{proof}[{Proof of~\Cref{claim:minor_properties}}]
    To address \cref{it:minor_4'}, let us assume for a contradiction that $v\in ((P_j \setminus T_j) \cup \{t_j\}) \cap T_{j'}$ for some $j,\,j' \in \{m-i+1, \ldots, m\}$, where $j\neq j'$. We recall that by \cref{it:minor_5}, $v\in T_{j'}$ and $t_j$ are at distance more than $24r(m-i)$ from each other.

    Also, we have
    \begin{equation*}
        \begin{split}
            \rho_{C_{m-i}}(s_{m-i}, u_{j'}) &\leq \rho_{C_{m-i}}(s_{m-i}, v) + \rho_{C_{m-i}}(v, u_{j'})\\
            &\leq \rho_{C_{m-i}}(s_{m-i}, v) + 24r \text.
        \end{split}
    \end{equation*}
    Similarly, decomposing the $\epsilon$-path $P_j$ into segments from $s_{m-i}$ over $v$ and $t_j$ up to $u_j$, we get
    \begin{equation*}
        \begin{split}
            \rho_{C_{m-i}}(s_{m-i}, u_j) &> \rho_{C_{m-i}}(s_{m-i}, v) + \rho_{C_{m-i}}(v, t_j) + 24r - \epsilon\\
            &> \rho_{C_{m-i}}(s_{m-i}, v) + 24r(m-i+1) - \epsilon \text.
        \end{split}
    \end{equation*}
    
    Consequently, we see that $|\rho_{C_{m-i}}(s_{m-i}, u_{j'}) - \rho_{C_{m-i}}(s_{m-i}, u_j)|\geq 24r(m-i) - \epsilon > 3r$. At the same time, there exists $n_{m-i} \in \mathbb N_0$ such that $u_j,\,u_{j'} \in C_{m-i+1} \subseteq A_{n_{m-i}} = \{v\in C_{m-i}\colon 3r(n_{m-i}-1)\leq \rho_{C_{m-i}}(v,s_{m-i})-\delta_{m-i} r < 3rn_{m-i}\}$. It follows that $|\rho_{C_{m-i}}(s_{m-i}, u_{j}) - \rho_{C_{m-i}}(s_{m-i}, u_{j'})|<3r$, which is absurd.

    As for the second part of \cref{it:minor_4'}, we can similarly verify that for any $v\in (P_j \setminus T_j) \cup \{t_j\}$, we have $\rho_{C_{m-i}}(s_{m-i}, u_j) - \rho_{C_{m-i}}(s_{m-i}, v)> 24r -\epsilon$. As in the already proven part, we recall that $|\rho_{C_{m-i}}(s_{m-i}, u_j)-\rho_{C_{m-i}}(s_{m-i}, w)| < 3r$ for any $w\in C_{m-i+1}$ by the assumption. That is, for any $v\in (P_j \setminus T_j) \cup \{t_j\}$ and $w\in C_{m-i+1}$ we get
    \begin{equation}\label{eq:minor_proof_distance1}
        \rho_{C_{m-i}} (s_{m-i}, w) - \rho_{C_{m-i}}(s_{m-i}, v) > 21r-\epsilon\text.
    \end{equation}
    In particular, we deduce that $(P_j \setminus T_j) \cup \{t_j\}$ and $\mathcal A_{j} \subseteq C_{m-i+1}$ are disjoint.

    In order to establish \cref{it:minor_7}, we can verify as above that for any $w\in C_{m-i+1}$ and $j\in \{m-i+1, \ldots, m\}$, we have
    \begin{equation*}
        \begin{split}
            \rho_{C_{m-i}} & (s_{m-i}, w) - \rho_{C_{m-i}}(s_{m-i}, h_{j}) \\ & > (\rho_{C_{m-i}}(s_{m-i}, u_{j}) - 3r) - (\rho_{C_{m-i}}(s_{m-i}, u_{j}) - 12r + \epsilon) = 9r-\epsilon.
        \end{split}
    \end{equation*}
    Clearly also $\rho_{C_{m-i}}(s_{m-i}, w) > 21r-\epsilon$ by \eqref{eq:minor_proof_distance1}.
    
    At the same time, we have $\rho_{C_{m-i}}(s_{m-i}, u_j) - \rho_{C_{m-i}}(s_{m-i}, h_{j}) \leq 12r$ for each $j\in \{m-i+1, \ldots, m\}$. Consequently, we obtain that $\rho_{C_{m-i}}(s_{m-i}, w) - \rho_{C_{m-i}}(s_{m-i}, h_{j}) < 15r$ for any $w\in C_{m-i+1}$. Similarly, it follows from \eqref{eq:minor_proof_distance1} that $\rho_{C_{m-i}}(s_{m-i}, w) - \rho_{C_{m-i}}(s_{m-i}, w') > 21r -\epsilon$ for any $w'\in (P_{j'} \setminus T_{j'}) \cup \{t_{j'}\}$, where $j'\in \{m-i+1, \ldots, m\}$. Subtracting the inequalities, we obtain $\rho_{C_{m-i}}(s_{m-i}, h_{j}) - \rho_{C_{m-i}}(s_{m-i}, w') > 6r - \epsilon$. This establishes \cref{it:minor_7}.
    \end{proof}
    
    Having established \Cref{claim:minor_properties}, we proceed to the inductive construction.

    Let $i=2$. We consider an $\epsilon$-path $p$ in $C_{m-1}$ connecting $s_{m-1}$ and $s_m$. We can easily see that $p$ can be cut into two subpaths of equal lengths, corresponding to sets $\mathcal A_{m-1}$ and $\mathcal A_{m}$ satisfying \cref{it:minor_1,it:minor_2,it:minor_3,it:minor_9}. Also, for each $j\in \{m-1, m\}$, we construct the tail $T_j \ni s_j$ such that it is subpath of length $24r$ of some $\epsilon$-path $P_j$ in $C_{m-2}$ connecting $s_j$ and $s_{m-2}$. Additionally, we let $t_j \in T_j$ denote the other endpoint of $T_j$, other than $s_j$, and let the middle point $h_j \in T_j$ be a point at distance $12r$ from $s_j$, along the tail $T_j$.

    It is easy to check that the tails $T_{m-1}$ and $T_m$ are disjoint. Specifically, note that if $u\in T_{m-1}$ and $v\in T_m$, then $\max\{ \rho(s_{m-1}, u), \rho(s_m, v) \} \leq 24r$, so that $\rho(u, v) \geq \rho(s_{m-1}, s_m) - 2\cdot 24r > 24r(m-2)$. Similarly, $\rho(h_{m-1}, h_m) > 24r(m-1)$ and for each $j\in \{m-1, m\}$, we have $\rho (s_{m-2}, h_j) \geq \rho(s_{m-2}, s_j) - \rho(s_j, h_j) > 24r(m-1)$. This establishes \cref{it:minor_5,it:minor_6}.
    
    To adress \cref{it:minor_4}, it is easy to see that since $s_m$ and $\mathcal A_{m-1}$ are at distance at least $24rm-\epsilon$ from each other in $\Sigma (G)$, the tail $T_m$ of length $24r$ from $s_m$ does not meet $\mathcal A_{m-1}$ (similarly for the role of $m$ and $m-1$ interchanged). We can also easily verify the validity of \cref{it:minor_8}.
    
    Altogether, we have verified that also the paths $P_j$, tails $T_j$ and the tips $t_j$ and middle points $h_j$, where $j\in \{m-1, m\}$, satisfy \cref{it:minor_4,it:minor_5,it:minor_6,it:minor_8}. This concludes the proof of the basis step.

    Let now $i \in \{2, \ldots, m-2\}$ be such that there exist supernodes $\mathcal A_j$, paths $P_j$, tails $T_j$, and tips $t_j$ and middle points $h_j$, where $j\in \{m-i+1, \ldots, m\}$, satisfying \cref{it:minor_1,it:minor_2,it:minor_3,it:minor_4,it:minor_5,it:minor_6,it:minor_8,it:minor_9} (and thus also \cref{it:minor_4',it:minor_7}).
    
    For each $j \in \{m-i+1, \ldots, m\}$, we let $\mathcal A'_{j} = \mathcal A_j \cup T_j$. We also define $\mathcal A'_{m-i}$ to be the union of all paths $(P_j \setminus T_j) \cup \{ t_j \}$, where $j \in \{m-i+1, \ldots, m\}$. By the induction hypothesis, it is easy to see that $\mathcal A'_{m-i}, \ldots, \mathcal A'_m$ satisfy \cref{it:minor_1,it:minor_2,it:minor_3}.
    
    To address \cref{it:minor_9}, we let $T'$ denote the set of all $x$ such that $x\in \mathcal A'_l \cap \mathcal A'_{l'}$ for some $l,\,l'\in \{m-i, \ldots, m\}$, where $l\neq l'$. Also, let $T$ be as in \cref{it:minor_9}. Note that for any $j \in \{m-i+1, \ldots, m\}$, we have $T' \cap \mathcal A'_j = (T \cap \mathcal A_j) \cup \{t_j\}$ by the construction. Moreover, $T \cap T_j = \emptyset$ by \cref{it:minor_4}. Consequently, it follows that $\mathcal A'_j \setminus T' = (\mathcal A_j \cup T_j) \setminus T' = (\mathcal A_j \setminus T) \cup (T_j \setminus \{ t_j \})$ is pathwise connected. Similarly, it follows from \cref{it:minor_4} that $\mathcal A'_{m-i} \setminus T' = \mathcal A'_{m-i} \setminus \{ t_j : j\in\{m-i+1, \ldots, m\}\} = \bigcup_{j\in\{m-i+1, \ldots, m\}} P_j \setminus T_j$ is pathwise connected.

    For each $j\in \{m-i, \ldots, m\}$, we denote $u_j=h_{j}$ (or $u_j = s_{m-i}$ for $j=m-i$) and construct an $\epsilon$-path $P'_j$ in $C_{m-i-1}$, connecting $u_j$ and $s_{m-i-1}$, such that $P'_j \cap C_{m-i}$ is a subpath of $P'_j$. In other words, there exists $v_0 \in P'_j$ such that $P'_j$ is the union of two paths: $p_1$ in $C_{m-i}$ connecting $u_j$ and $v_0$, and $p_2$ from $v_0$ to $s_{m-i-1}$. Moreover, $p_2\setminus\{v_0\}\subset C_{m-i-1}\setminus C_{m-i}$. 

    Note that we impose an additional condition on $P'_j$, compared to the basis step, to ensure that $T'_j$ is disjoint from $\mathcal A'_{j'}$, where $j'\neq j$. Indeed, this is because we generally know less about the distance from $u_j$ to the set $\mathcal A'_{j'}$.

    \begin{claim}\label{claim:minor_path}
        There exists a path $P'_j$ with the desired properties.
    \end{claim}
    
    \begin{proof}[{Proof of~\Cref{claim:minor_path}}]
    To construct such a path $P_{j'}$, we consider $n_{m-i-1}\in \mathbb N_0$ satisfying that $\rho_{C_{m-i-1}}(s_{m-i-1}, u_j) - \delta_{m-i-1}r \in [3(n_{m-i-1}-1)r, 3n_{m-i-1}r)$, and put $\epsilon' = \min \{ \epsilon, 3n_{m-i-1}r + \delta_{m-i-1}r - \rho_{C_{m-i-1}}(s_{m-i-1}, u_j) \}$. Subsequently, we let $p$ be an $\epsilon'$-path from $u_j$ to $s_{m-i-1}$. Note that by the choice of $\epsilon'$, we have
    \begin{equation}\label{eq:minor_v_distance}
        \rho_{C_{m-i-1}}(s_{m-i-1}, v) - \delta_{m-i-1}r < 3n_{m-i-1}r \text{,\quad where $v\in p$\text.}
    \end{equation}

    We claim that $s_{m-i-1}\notin C_{m-i}$. Indeed, note that by the assumption, we have $\rho(s_{m-i-1}, s_{m-i}) > 24rm$, and at the same time, for any $v \in C_{m-i}$ it holds that $|\rho_{C_{m-i-1}} (s_{m-i-1}, s_{m-i}) - \rho_{C_{m-i-1}} (s_{m-i-1}, v)| < 3r$ by the assumption that $C_{m-i}$ is a connected component in $A_{n_{m-i-1}} = \{v\in C_{m-i-1}\colon 3r(n_{m-i-1}-1)\leq \rho_{C_{m-i-1}}(v,s_{m-i-1})-\delta_{m-i-1} r < 3rn_{m-i-1}\}$.
    
    Along the path $p$, starting from $u_j$ and going to $s_{m-i-1}\notin C_{m-i}$, let $e$ be the first edge such that for some $v \in e \cap p$, we have $v\notin C_{m-i}$. In particular, it follows from \eqref{eq:minor_v_distance} that $\rho_{C_{m-i-1}}(s_{m-i-1}, v) - \delta_{m-i-1}r < 3(n_{m-i-1}-1)r$. By the assumption that $C_{m-i}$ is a pathwise connected component of $A_{n_{m-i-1}} = \{v\in C_{m-i-1}\colon 3r(n_{m-i-1}-1)\leq \rho_{C_{m-i-1}}(v,s_{m-i-1})-\delta_{m-i-1} r < 3rn_{m-i-1}\}$, we see that $e \cap C_{m-i}$ splits into at most two pathwise connected components. Hence, going from $u_j$ to $v$ along $p$, we see that this path further splits into two pathwise connected subsets, the first of which is contained in $C_{m-i}$ and the other is not.
    
    We denote $\epsilon'' = \min\{\epsilon, 3(n_{m-i-1}-1) + \delta_{m-i-1}r - \rho_{C_{m-i-1}}(s_{m-i-1}, v) \}$ and modify $p$ so that from $v$, it continues as some $\epsilon''$-path $p'$ in $C_{m-i-1}$ from $v$ to $s_{m-i-1}$. Without loss of generality, we can assume that the modified path does not intersect itself. Moreover, by the choice of $\epsilon''$, the path $p'$ does not intersect $C_{m-i}$. Hence, by the construction, the modified path $p$ has the property that $p\cap C_{m-i}$ is a subpath of $p$. We let $P'_j = p$.
    \end{proof}

    Having established the existence of a path $P'_j$ with the desired properties, we let the tail $T'_j \ni u_j$ be a subpath of length $24r$ of $P'_j$. Additionally, we let $t'_j \in T'_j$ denote the other endpoint of $T'_j$, other than $u_j$, and the middle point $h'_j \in T'_j$ be the point at distance $12r$ from $u_j$, along the tail $T'_j$. It remains to verify \cref{it:minor_4,it:minor_5,it:minor_6,it:minor_8}.

    Using the induction hypothesis that the middle points $h_j$, where $j\in\{m-i+1, \ldots, m\}$, and $s_{m-i}$ are more than $24r(m-i+1)$ far from each other in $C_{m-i}$ (see \cref{it:minor_5,it:minor_6}), we observe similarly as in the basis step that \cref{it:minor_5} holds for the new tails $T'_j$ and midpoints $h'_j$ (in particular, the new tails $T'_j$ are disjoint).
    
    Similarly, we obtain that for each $j\in\{m-i, \ldots, m\}$, the tail $T'_j$ is disjoint from $T_{j'}$ for any $j'\in\{m-i+1, \ldots, m\}$, where $j'\neq j$.
    
    More generally, to establish \cref{it:minor_4}, we wish to show that $\mathcal A'_{j}$ and $T'_{j'}$ are disjoint for each $j,\,j' \in \{m-i, \ldots, m\}$, where $j\neq j'$. To that end, assume for a contradiction that $\mathcal A'_{j}$ and $T'_{j'}$ meet at some point point $v$. As before, we denote $u_{j'}=h_{j'}$ (or $s_{m-i}$ for $j'=m-i$). We recall that $u_{j'},\,v\in C_{m-i}$ and, consequently, $\rho_{C_{m-i-1}} (s_{m-i-1}, u_{j'}) - \rho_{C_{m-i-1}} (s_{m-i-1}, v) < 3r$.
    
    Also, if $\ell$ denotes the length of the subpath of $T'_{j'}$ connecting $u_{j'}$ and $v$ (this subpath lies in $C_{m-i}$ by the construction of $P'_{j'}$, i.e., $\ell \geq \rho_{C_{m-i}} (u_{j'}, v) \geq |\rho_{C_{m-i}}(s_{m-i}, h_{j'}) - \rho_{C_{m-i}}(s_{m-i}, v)|$), we have that $\rho_{C_{m-i-1}} (s_{m-i-1}, u_{j'}) > \ell + \rho_{C_{m-i-1}} (s_{m-i-1}, v) - \epsilon$, that is, $\rho_{C_{m-i-1}} (s_{m-i-1}, u_{j'}) -  \rho_{C_{m-i-1}} (s_{m-i-1}, v) + \epsilon > \ell $. Altogether, we get that 
    \begin{equation}\label{eq:minor_proof_distance_s_v}
        3r+\epsilon > \ell \geq |\rho_{C_{m-i}}(s_{m-i}, h_{j'}) - \rho_{C_{m-i}}(s_{m-i}, v)| \text.
    \end{equation}
  
    We note that if $j<m-i$, then necessarily $v\in \mathcal A_j \subseteq C_{m-i+1}$, because $\mathcal A'_j = \mathcal A_j \cup T_j$ and $T'_{j'} \ni v$ and $T_j$ are disjoint by the already proven part. Otherwise if $j=m-i$, then $v\in (P_{l} \setminus T_{l}) \cup \{t_{l}\}$ for some $l \in \{m-i+1, \ldots, m\}$ by the construction. Consequently, in both cases ($j<m-i$ and $j=m-i$), \cref{it:minor_7} shows that $|\rho_{C_{m-i}}(s_{m-i}, h_{j'}) - \rho_{C_{m-i}}(s_{m-i}, v)| > 6r-\epsilon$, which is impossible by \eqref{eq:minor_proof_distance_s_v}. This contradiction finishes the proof that \cref{it:minor_4} holds.

    As for \cref{it:minor_8}, we claim that $\rho(s_j, h'_j) \leq 12ri$ for each $j\in\{m-i, \ldots, m\}$. Indeed, for $j> m-i$ this follows from inductive hypothesis and the fact that $\rho(h'_j, h_j)\leq 12r$. Moreover, we clearly have $\rho(s_{m-i}, h'_{m-i}) \leq 12r$.
    
    As it also holds that $\rho(s_{m-i-1}, s_{j}) > 24rm$ for any $j\in\{m-i, \ldots, m\}$ by the assumption, we deduce that $\rho(s_{m-i-1}, h'_{j}) > 24r(m-i)$ for any such $j$, which verifies \cref{it:minor_6}.

    Altogether, we have verified that also the paths $P_j$, tails $T_j$ and the tips $t_j$ and middle points $h_j$, where $j\in \{m-i,\ldots, m\}$, satisfy \cref{it:minor_4,it:minor_5,it:minor_6,it:minor_8}. This concludes the proof of the inductive step.

    Consequently, we can find supernodes $\mathcal A_i$, paths $P_i$, tails $T_i$, and tips $t_i$ and middle points $h_i$, where $i\in \{2, \ldots, m\}$, satisfying \cref{it:minor_1,it:minor_2,it:minor_3,it:minor_4,it:minor_4',it:minor_5,it:minor_6,it:minor_7,it:minor_8,it:minor_9}. Then, as in the first part of the proof of the induction step, we construct supernodes $\mathcal A_i$, where $i\in \{1, \ldots, m\}$, satisfying \cref{it:minor_1,it:minor_2,it:minor_3,it:minor_9}.

    We claim that $K_m$ is a minor of $G$. Indeed, if $m=3$, it is clear that we find a cycle in $G$. For $m>3$, we inductively define $\mathcal A'_i = \mathcal A_i \cap V \setminus \bigcup_{j=1}^{i-1} \mathcal A'_j$, where $i\in \{1, \ldots, m\}$. Note it follows from \cref{it:minor_9} that the vertex sets $\mathcal A'_i$ are non-empty connected subgraphs in $G$ and moreover, by the definition of $\mathcal A'_i$, any two vertex sets are connected by an edge in $E$ (consider the cases when two distinct supernodes meet at a vertex of $G$ or within the interior of some edge). The claim follows.
\end{proof}

\emergencystretch=0.5em
\printbibliography

\end{document}